\documentclass[11pt]{article}
\usepackage{amsmath, amssymb, amsthm, graphicx}
\usepackage[makeroom]{cancel}
\usepackage{tcolorbox}
\usepackage[colorlinks=true,linkcolor=blue,citecolor=blue]{hyperref}
\usepackage{tikz-cd}
\usepackage{geometry}
\usepackage{mathrsfs}  
\usepackage{subfig}

\newcommand{\C}{\mathcal{C}}
\newcommand{\F}{\mathcal{F}}

\newcommand{\K}{\mathcal{K}}

\newcommand{\R}{\mathbb{R}}

\newcommand{\N}{\mathbb{N}}
\newcommand{\W}{\mathcal{W}}
\newcommand{\Z}{\mathbb{Z}}

\newcommand{\Nei}{\mathcal{N}}
\newcommand{\dom}{\mathrm{dom}}

\newcommand{\diam}{\mathop{\mathrm{diam}}}

\newcommand{\floor}[1]{\lfloor #1 \rfloor}
\newcommand{\ceil}[1]{\lceil #1 \rceil}

\renewcommand{\int}[1]{\mathrm{int(#1)}}

\newcommand{\im}[1]{\mathrm{im}(#1)}

\newtheorem{theorem}{Theorem}[section]
\newtheorem{lemma}{Lemma}[section]

\newtheorem{proposition}{Proposition}[section]

\theoremstyle{definition}
\newtheorem{definition}{Definition}[section]

\theoremstyle{remark}
\newtheorem{example}{Example}[section]
\newtheorem{remark}{Remark}[section]

\newcommand{\Zb}{\overline{\Z}}
\newcommand{\smalltile}[1]{\raisebox{-1mm}{\includegraphics{Tiles/smalltile#1}}}
\newcommand{\pattern}[1]{\begin{smallmatrix}#1\end{smallmatrix}}
\newcommand{\whitetile}{\raisebox{-.5mm}{\includegraphics{Tiles/tiles-7}}}
\newcommand{\blacktile}{\raisebox{-.5mm}{\includegraphics{Tiles/tiles-8}}}

\newcommand{\gen}{\mathscr{L}^0}
\newcommand{\nar}{\mathscr{L}^1}

\newcommand{\twin}[1]{\overline{#1}^a}

\definecolor{mypurple}{rgb}{1, .5, 1}
\definecolor{myred}{rgb}{1, .5, .5}
\definecolor{myblue}{rgb}{.5, .5, 1}
\definecolor{mygreen}{rgb}{.5, 1, .5}

\definecolor{darkpurple}{rgb}{.8, .4, .8}
\definecolor{darkred}{rgb}{.8, .4, .4}
\definecolor{darkblue}{rgb}{.4, .4, .8}
\definecolor{darkgreen}{rgb}{.4, .8, .4}

\newcommand{\purple}{\textcolor{darkpurple}{purple}}
\newcommand{\red}{\textcolor{darkred}{red}}

\newcommand{\green}{\textcolor{darkgreen}{green}}

\title{Local generation of tilings\footnote{The previous version from November 13, 2024 contained a mistake: the proof of Proposition 5.6 was incorrect. In this version, the statement and proof are corrected.}}
\author{Juliette Favereau\footnote{\'Ecole des Mines de Nancy} ~and Mathieu Hoyrup\footnote{Universit\'e de Lorraine, CNRS, Inria, LORIA, Nancy, 54000, France}}
\date{December 3, 2025}

\begin{document}
\maketitle

\begin{abstract}
In this article, we investigate the possibility of generating all the configurations of a subshift in a local way. We propose two definitions of local generation, explore their properties and develop techniques to determine whether a subshift satisfies these definitions. We illustrate the results with several examples.
\end{abstract}

\tableofcontents

\ifdefined\addparentheses
\newcommand{\ti}[1]{(#1)}
\else
\newcommand{\ti}[1]{#1}
\fi

%SSSSSS
\section{Introduction}
If ones fixes a Wang tileset, then the problem of generating a tiling of a plane is difficult in general. In the worst case, Hanf and Myers \cite{Hanf74,Myers74} showed the existence of tilesets that tile the plane, but for which no algorithm can produce any tiling. %This result witnesses in particular the fact that local rules can induce long-range correlations. % and that local matching rules do not in general imply local construction rules.
On the other end of the spectrum, certain tilesets admit a simple procedure that can generate any tiling of the plane, in which the choices of the tiles in each cell only require local interaction (we give an example in the next paragraphs).

In this article, we investigate the problem, given a fixed Wang tileset or a more general subshift, of generating all the tilings of the plane in a local way, i.e.~with a minimal amount of communication between the cells. Of course, one expects that only very few tilesets admit such local generation procedures. Our goal is to give precise definitions capturing this idea of local generation, and then to settle whether each given tileset or subshift satisfies these definitions. Let us stress that our perspective is orthogonal to the viewpoints provided by computability and complexity theory, and that the locality of a procedure is a form of combinatorial rather than computational simplicity.

Let us start by illustrating a simple local generation procedure on a motivating example, which is the Wang tileset shown in Figure \ref{fig_boundaries}. It is the tileset with two colors (on the picture, an edge is either white or crossed by a black line) containing all the tiles having an even number of each color.
\begin{figure}[!ht]
\centering
\includegraphics{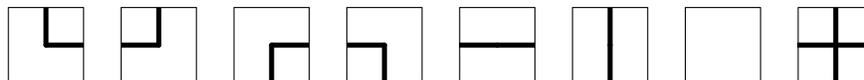}
\caption[A tileset admitting a local generation procedure]{A tileset admitting a local generation procedure}\label{fig_boundaries}
\end{figure}

There is a simple local way of generating any tiling of the plane. First assign arbitrary bits to the points of the grid~$\Z^2$, then for each edge~$(a,b)$ between two adjacent points, draw a black line crossing that edge if and only if the bits assigned to~$a$ and~$b$ differ. Intuitively, the bit assignment partitions~$\Z^2$ into connected components on which the assignment is constant, and the black lines are the boundaries between these regions. %interpret two adjacent points with the same bits as being part of a common region, and draw boundaries between different regions. In more concrete terms, for each edge~$(a,b)$ between two adjacent points, draw a line segment crossing that edge precisely when the bits assigned to~$a$ and~$b$ are different.
This construction is local in the sense that if two cells are not neighbors (horizontally, vertically or diagonally), then the choices of the tiles that are placed in those cells are made independently. The procedure that generates a tiling takes as input an arbitrary assignment of bits to the points of the grid, each cell then reads the bits assigned to its $4$ corners and determines its tile. Note that this procedure enables adjacent cells to communicate in the sense that they share some input information. Similar procedures are used in computer graphics to generate textures in an efficient and local way \cite{Lagae09}.
%
%\begin{figure}[!ht]
%\centering
%\includegraphics{Tilesets/tileset_corners-0}
%\caption[Tiles determined by corners]{Tiles determined by the colors of the corners}\label{fig_bicolor}
%\end{figure}
%
%A simplified version consists in forgetting the colors and only keeping the boundaries between the regions having different colors, inducing the tileset of size~$8$ shown in Figure \ref{fig_boundaries}.
%
%\begin{figure}[!ht]
%\centering
%\includegraphics{even-28}
%\caption[Tiles determined by corners]{Tiles determined by the colors of the corners: boundaries}\label{fig_boundaries}
%\end{figure}
%
%These tilesets are built so that tilings can be easily generated, in a local way: if two cells are not neighbors (horizontally, vertically or diagonally), then the choices of the tiles that are placed in these cells are made independently. The procedure that generates a tiling takes as input an arbitrary assignment of colors to the points of the grid, each cell reads the colors of its $4$ corners and determines its tile. Note that this procedure enables adjacent cells to ``communicate'' in the sense that they share some input information.

%It produces only valid tilings, and can produce all of them. In other words, it is a surjective function from~$\{0,1\}^{\Z^2}$ to the set of valid tilings, which is continuous (the choice of a tile in a cell only depends on a finite number of input bits) and commutes with the shift action (all the cells have the same rule for choosing the tile from the input bits). Such functions are called sliding block codes in the literature \cite{LM95}.

A first natural definition, that would capture this generation procedure % associated to the tileset in Figure \ref{fig_boundaries}, %tilesets in Figures \ref{fig_bicolor} and \ref{fig_boundaries}, 
would be that the subshift is a factor of a fullshift, i.e.~the image of a fullshift by a so-called sliding block code \cite{LM95}, which is a finite rule determining the output content of a cell from the input content of a neighborhood of that cell. %continuous function that commutes with the shift.
Indeed, the set of all possible assignments of bits is the fullshift~$\{0,1\}^{\Z^2}$, and the function that converts this color assignment to the assignment of tiles to the cells can be evaluated in a local way. Equivalently, a sliding block code is a continuous function that commutes with the shift and can be seen as cellular automata with possibly different input and output alphabets.

However, this tentative definition would be very restrictive in several ways:
\begin{itemize}
\item Only nearby cells can communicate, i.e.~share input information,
\item All the cells are required to have the same rule.
\end{itemize}
We will relax these restrictions in our definitions. However, there is an important feature that we want to keep:
\begin{itemize}
\item The content of each cell is determined by an input region whose \emph{size} is uniformly bounded across the cells.
\end{itemize}
When this restriction is met, it gives rise to a procedure in which all the cells can determine their contents in parallel, reading a constant number of input symbols, so this procedure runs in constant parallel time.

We introduce two possible formalizations of the intuitive idea of local generation by defining two classes of subshifts~$\gen\subseteq\nar$. The class~$\gen$ is defined as the smallest class of subshifts containing certain basic subshifts (to be defined) including fullshifts, and closed under finite products and factors. The class~$\nar$ is based on the notion of \emph{narrow} function which is a relaxation of sliding block codes, in which the content of each output cell is determined by the content of a finite region of input cells whose size is bounded, but whose shape is arbitrary. The definition of~$\nar$ is purely combinatorial in that it does not interact with the dynamical structure of the subshift. We never address the computability or complexity of the underlying procedures, and only focus on some sort of combinatorial complexity.

We develop the theories of the classes~$\gen$ and~$\nar$, identify two obstructions to being in these classes, and apply these techniques to a number of subshifts.  We also show that one-dimensional subshifts of finite type (SFTs) all belong to~$\gen$. It is not surprising as those subshifts are known to have many good properties. In particular, their configurations are nothing else than bi-infinite walks in a finite graph, which are intuitively easy to generate; the result shows that the class~$\gen$ is sufficiently rich to capture this intuition. In a forthcoming article \cite{FH24b}, we apply the techniques develop in this article to classify all the Wang tilesets made with bicolor tiles having an even number of each color. We briefly summarize this classification at the end of this article.

The study of the class~$\gen$ is related to the problem of expressing a subshift as a factor of another subshift, which was investigated in \cite{Marcus79,Boyle83,BPS10,BMP18}, but our results are rather orthogonal.

%The companion article \cite{FH24b} classifies all the Wang tilesets with two colors and an even number of edges of each color, using the techniques developed in this article. We decided to present this classification separately because the present article is already long enough.

The article is organized as follows. In Section \ref{sec_def} we present the definition of the two classes of subshifts~$\gen$ and~$\nar$. In Section \ref{sec_dynamics} we introduce the notions from symbolic dynamics that will play an important role in our analysis, recalling classical definitions and stating useful properties. In Section \ref{sec_properties} we investigate several properties of our two classes. In Section \ref{sec_obstructions} we develop a technique to prove that a subshift does not belong to~$\gen$, and apply it to a number of subshifts. We carry out an analogous development for~$\nar$ in Section \ref{sec_transitions}, using different arguments. In Section \ref{sec_positive}, we identify several subshifts that belong to~$\gen$, among which the one-dimensional SFTs. In Section \ref{sec_classification}, we summarize the classification of the even bicolor tilesets. We end with Section \ref{sec_future} where we suggest future research directions. We also include an appendix, Section \ref{sec_back}, containing classical results that are used in the proofs.

%TODO
%Note that in a sliding block code, each the input region that determines the content of a cell is a neighborhood of the cell with a fixed \emph{diameter}. Our definitions will allow the input region to have more flexible shapes and will require that their sizes are bounded, rather than their diameters. A simple example is when the content of a cell~$p\in\Z^d$ depends on the content of the same input cell~$p$ as well as the cell~$0$ (the origin). The input region of a cell~$p$ is then~$\{0,p\}$, which has size at most~$2$ but unbounded diameter. This situation corresponds to the case when all the cells need to communicate together, by reading the content of the origin, and then make their choices independently.

%We are going to introduce two notions of local generation procedures, which can be informally described as follows:
%\begin{itemize}
%\item Communication is allowed along a finite number of directions (for instance, all the cells in a row share some input information) and all the cells have the same rule,
%\item The input regions of the cells have bounded size, but there is no assumption about their shapes and the associated rules.
%\end{itemize}

%SSSSSS
\section{Local generation}\label{sec_def}

Let us first recall the minimal amount of definitions needed to introduce our notions of local generation.
%ssss
\subsection{Preliminaries}

For~$d\geq 1$,~$(\Z^d,+,0)$ is an abelian group. A~$\Z^d$\textbf{-action} on a set~$E$ is a homomorphism from~$\Z^d$ to the group of bijections from~$E$ to~$E$ with the composition operation. We also say that~$E$ is a~$\Z^d$\textbf{-set}. The result of applying the image of~$p\in\Z^d$ under the homomorphism to~$e\in E$ is written as~$p\cdot e\in E$. A \textbf{continuous~$\Z^d$-action} on a topological space~$X$ is a homomorphism from~$\Z^d$ to the group of homeomorphisms from~$X$ to~$X$ with the composition operation. We also say that~$X$ is a~$\Z^d$\textbf{-space}. A~$\Z^d$-space~$Y$ is a \textbf{factor} of a~$\Z^d$-space~$X$ if there exists a continuous surjective map~$f:X\to Y$ that commutes with the actions:~$p\cdot f(x)=f(p\cdot x)$. $f$ is called a \textbf{factor map}.

If~$A$ is a finite alphabet and~$E$ is countable, then~$A^E$ is endowed with the Cantor topology, which is the product of the discrete topology on~$A$. This makes~$A^E$ a compact metrizable space. The \textbf{shift action} on~$A^{\Z^d}$ is the continuous~$\Z^d$-action defined, for~$p\in\Z^d$ and~$x\in A^{\Z^d}$, by~$p\cdot x=y$ where~$y(q)=x(p+q)$. A~\textbf{$\Z^d$-subshift} is a compact set~$X\subseteq A^{\Z^d}$ which is shift-invariant, i.e.~satisfies~$\sigma^p(X)=X$ for all~$p\in\Z^d$. A~\textbf{$\Z^d$-fullshift} is~$A^{\Z^d}$ for some finite alphabet~$A$.

%ssss
\subsection{The class \texorpdfstring{$\gen$}{C0}}
We start with the strongest notion of local generation. The definition is motivated by the following observations:
\begin{itemize}
\item Certain particular subshifts have a generation procedure that is intuitively local,
\item If two subshifts~$X$ and~$Y$ can be locally generated, then so is their product~$X\times Y$ (which is indeed a subshift), because the content of a cell is simply a pair whose components can be locally generated,
\item If a subshift~$X$ can be locally generated, then so is any factor~$Y$ of~$X$. Indeed, if~$f:X\to Y$ is a factor map then the value of~$f(x)_p$ is determined by the values of~$x$ in finitely many cells around~$p$, which in turn can be locally generated.
\end{itemize}
We then define the class~$\gen$ as the smallest class of subshifts containing the basic subshifts from the first item, and that is closed under finite products and factors. We first need to choose basic subshifts to start with. We propose the following three families and explain why they can be considered as locally generated:
\begin{itemize}
\item \textbf{Fullshifts}~$A^{\Z^d}$: the contents of the cells can be chosen independently of each other.
\item \textbf{Periodic shifts}: given a subgroup~$H\subseteq\Z^d$, let~$X_H\subseteq A^{\Z^d}$ be the subshift containing all the~$H$-periodic configurations, i.e.~$x\in X_H$ if~$x(p+h)=x(p)$ for all~$p\in\Z^d$ and~$h\in H$. First, an element of~$A^{\Z^d/H}$ can be generated by independent choices, and then the content of a cell~$p\in\Z^d$ is determined by the symbol assigned to its equivalence class in~$\Z^d/H$. We will denote~$X_H$ by~$A^{\Z^d/H}$.
\item \textbf{Countable subshifts}: the configurations of a countable subshift can be indexed by~$\N$, so a configuration can be produced by first choosing a natural number~$n$ shared by all the cells, and then all the cells determine their contents in parallel. We do not address the computability or complexity of this procedure (it was shown in \cite{ST13} that countable subshifts can be very uncomputable).
\end{itemize}

\begin{definition}[\ti{The class $\gen$}]
Let~$d\geq 1$. We define~$\gen_d$ as the smallest class of subshifts containing the~$\Z^d$-fullshifts, the periodic~$\Z^d$-shifts and the countable~$\Z^d$-subshifts, and which is closed under finite products and factors. Let~$\gen=\bigsqcup_{d\geq 1}\gen_d$.
\end{definition}

First, subshifts in~$\gen$ can be presented in a simple form.
\begin{proposition}[\ti{Canonical form}]\label{prop_normal_form}
One has~$X\in\gen$ if and only if~$X$ is a factor of a finite product of basic subshifts. Moreover, we can assume that this finite product is
\begin{equation*}
K\times A^{\Z^d}\times A^{\Z^d/H_1}\times \ldots\times A^{\Z^d/H_k}
\end{equation*}
where~$K$ is countable,~$A$ is finite and~$H_1,\ldots,H_k$ are non-trivial subgroups of~$\Z^d$ of ranks at most~$d-1$ (and possibly~$k=0$).
\end{proposition}
\begin{proof}
We first show that every~$X\in\gen$ is a factor of a finite product of basic subshifts, by induction on the finite tree describing the construction of~$X$. If~$X$ is a basic subshift, then there is nothing to do. If~$X=X_0\times X_1$ with~$X_0,X_1\in\gen$ then by induction~$X_0=f_0(P_0)$ and~$X_1=f_1(P_1)$ where~$f_0,f_1$ are factor maps and~$P_0,P_1$ are finite products of basic subshifts. Therefore~$X=f(P_0\times P_1)$ where~$f(x_0,x_1)=(f_0(x_0),f_1(x_1))$, and note that~$f$ is a factor map and~$P_0\times P_1$ is a finite product of basic subshifts. If~$X=f(Y)$ where~$Y\in\gen$ and~$f$ is a factor map, then by induction~$Y=g(P)$ for some finite product~$P$ of basic subshifts and some factor map~$g$, so~$x=f\circ g(P)$ and~$f\circ g$ is a factor map.

Finally, a finite product of countable subshifts is a countable subshift, a finite product of fullshifts is a fullshift, and one can replace all the involved alphabets by the largest one because if~$A\subseteq B$ and~$H\subseteq\Z^d$, then~$A^{\Z^d/H}$ is a factor of~$B^{\Z^d/H}$. If~$H$ is trivial then~$A^{\Z^d/H}$ is a fullshift. If~$H$ has rank~$d$, then~$A^{\Z^d/H}$ is finite hence countable so it can be merged with~$K$.
\end{proof}

Note that the presence of a periodic shift~$A^{\Z^d/H}$ in the canonical form intuitively allows output cells that differ by an element of~$H$ to ``communicate'', in the sense that they share some amount of input information and can make their choices in tandem with each other.

%\begin{definition}
%Let~$d\geq 1$ and~$\Sigma$ be a finite alphabet. A~$\Z^d$-subshift~$X\subseteq \Sigma^{\Z^d}$ belongs to~$\gen$ if there exists a countable compact~$\Z^d$-space~$K$, a countable discrete~$\Z^d$-set~$E$ and a finite alphabet~$A$ such that~$X$ is a factor of~$K\times A^E$.
%\end{definition}

It will be convenient to present a canonical form in a more condensed way, as follows.
\begin{proposition}\label{prop_normal_form_bis}
One has~$X\in\gen$ if and only if~$X$ is a factor of~$K\times A^E$ where~$K$ is a countable subshift,~$A$ is finite and~$E$ is a countable~$\Z^d$-set.
\end{proposition}
It is understood that~$A^E$ is endowed with the~$\Z^d$-action~$(u\cdot a)_p=a_{u\cdot p}$ for~$a\in A^E$, and that~$K\times A^E$ is endowed with the product~$\Z^d$-action~$u\cdot (k,a)=(u\cdot k,u\cdot a)$.

\begin{proof}
If~$X\in\gen$ then by Proposition \ref{prop_normal_form},~$X$ is a factor of~$K\times A^{\Z^d}\times A^{\Z^d/H_1}\times\ldots\times A^{\Z^d/H_k}=K\times A^E$ where~$E=\Z^d\sqcup \Z^d/H_1\sqcup\ldots\sqcup \Z^d/H_k$ endowed with the natural~$\Z^d$-action.

Conversely, if~$X$ is a factor of~$K\times A^E$ for some countable~$\Z^d$-set~$E$, then by the orbit-stabilizer theorem (Theorem \ref{thm_orbit_stab}), one has~$E=\bigsqcup_{i\in I}\Z^d/H_i$ where~$I$ is countable and~$H_i$ is a subgroup of~$\Z^d$. By compactness of~$K\times A^E$ and continuity of~$f$, there exists a finite set~$R\subseteq E$ such that~$f(k,x)_0$ is determined by the values of~$x$ on~$R$. As~$f$ is a factor map, for any~$p\in\Z^d$,~$f(k,x)_p$ is determined by the values of~$x$ on~$p\cdot R=\{p\cdot e:e\in R\}$. Therefore,~$f(k,x)$ is determined by the values of~$x$ on the obits of elements of~$R$. As~$R$ is finite, there are finitely many such orbits, so~$X$ is a factor of~$K\times A^{E'}$ where~$E'$ is a finite disjoint union of quotients of~$\Z^d$, so~$A^{E'}$ is a finite product of~$A^{\Z^d/H_i}$.
\end{proof}

\begin{example}[\ti{Cellular automaton}]
A cellular automaton is a continuous function~$f:A^{\Z^d}\to A^{\Z^d}$ which commutes with the shift action, so its image is a factor of a fullshift and always belongs to~$\gen$.
\end{example}

%ssss
\subsection{The class \texorpdfstring{$\nar$}{nar}}
We present the second definition of local generation, which is a relaxation of the first one in which we forget the dynamical aspects of subshifts. We first need the following notions in order to formulate the definition of~$\nar$.

Let~$E,F$ be countable sets,~$A,B$ be finite alphabets and~$f:A^E\to B^F$ be continuous. The following notion is a restriction on the continuity of~$f$, in which each output position has a narrow view on the input of~$f$.
\begin{definition}[\ti{Narrow function}]
We say that~$q\in F$ is \textbf{determined} by~$R\subseteq E$ if for all inputs~$x,y\in A^E$ that coincide on~$R$,~$f(x)$ and~$f(y)$ coincide at position~$q$.

Let~$r\in\N$. The function~$f:A^E\to B^F$ is~$r$\textbf{-narrow} if each~$q\in F$ is determined by a region of size at most~$r$. We say that~$f$ is \textbf{narrow} if~$f$ is~$r$-narrow for some~$r$.
\end{definition}

Note that a function is~$0$-narrow if and only if it is constant. The identity~$f:A^E\to A^E$ is~$1$-narrow. Every factor map is narrow. The composition of two narrow functions is narrow. 
%It should be contrasted with the following observation: for any subshift, more generally any compact subset~$X$ of a Cantor space, there exists a continuous surjective function~$f:\{0,1\}^\N\to X$. Such a function induces a procedure that runs in parallel for all the cells, where each cell reads a finite number of input bits, but in which that finite number of bits is \emph{not bounded} across the cells.%Every continuous function~$f:A^{\Z^d}\to B^{\Z^d}$ that commutes with the shift action is narrow, because by Curtis-Hedlund-Lyndon theorem \cite[Theorem 6.2.9]{LM95},~$f$ is a sliding block code, i.e.~$\W_f(q)$ is contained in the neighborhood~$\Nei(q,r)=\{p\in\Z^d:d(p,q)\leq r\}$ for some~$r$ (the metric~$d$ is defined by~$d(p,q)=\max_i |p_i-q_i|$ where~$p_i,q_i$ are the coordinates of~$p,q$ respectively).

\begin{example}[\ti{Non-uniform cellular automaton}]\label{ex_non-uniform_CA}
A non-uniform cellular automaton \cite{DFP12} is a continuous function~$f:A^{\Z^d}\to A^{\Z^d}$ such that each~$q\in\Z^d$ is determined by~$\Nei(q,r)=\{p\in\Z^d:d(p,q)\leq r\}$ for some~$r\in\N$ (the metric~$d$ is defined by~$d(p,q)=\max_i |p_i-q_i|$ where~$p_i,q_i$ are the coordinates of~$p,q$ respectively). The size of~$\Nei(q,r)$ is constant so~$f$ is narrow.
\end{example}

We come to the second notion of local generation. Let~$\Sigma$ be finite and~$F$ be countable.

\begin{definition}[\ti{The class $\nar$}]
A compact set~$X\subseteq\Sigma^F$ belongs to~$\nar$ if it is a countable union of images of narrow functions, i.e.~$X=\bigcup_{n\in\N} X_n$ where~$X_n=\im{f_n}$ and~$f_n:A_n^\N\to X$ is~$r_n$-narrow.
\end{definition}
Note that~$\N$ can be replaced by any countable set. We will mostly consider the case when~$F=\Z^d$ and~$X$ is a~$\Z^d$-subshift, but the definition makes sense more generally.

Finally, this notion of local generation is indeed a relaxation of the previous one.
\begin{proposition}\label{prop_inclusion}
One has~$\gen\subseteq \nar$.
\end{proposition}
\begin{proof}
Let~$X\in\gen$, let~$f:K\times A^E\to X$ be a factor map given by Proposition \ref{prop_normal_form_bis}. For each~$k\in K$, let~$f_k:A^E\to X$ send~$a$ to~$f(k,a)$. As~$f$ is narrow, each~$f_k$ is narrow so~$X=\bigcup_{k\in K}\im{f_k}$ belongs to~$\nar$, as~$K$ is countable.
\end{proof}

The following notion will be useful to analyze narrow functions.
\begin{proposition}[Input window]
Let~$f:A^E\to B^F$ be continuous. For each~$q\in F$, there is a minimal subset of~$E$ that determines~$q$. It is denoted by~$\W_f(q)$ and is called the \textbf{input window} of~$q$.
\end{proposition}
\begin{proof}
Let~$\W_f(q)$ be the intersection of all the sets that determine~$q$. We show that~$\W_f(q)$ determines~$q$. Let~$x,y\in A^E$ coincide on~$\W_f(q)$. As~$f$ is continuous, there is a finite set~$R$ that determines~$q$, so~$\W_f(q)\subseteq R$. Let~$R\setminus \W_f(q)=\{r_1,\ldots,r_n\}$. For each~$i\leq n$, there exists~$R_i$ that determines~$q$, and such that~$r_i\notin R_i$. We inductively define~$x_0,x_1,\ldots,x_n$ as follows. Let~$x_0=x$ and for~$i<n$, let~$x_{i+1}$ coincide with~$x_i$ everywhere except possibly at~$r_i$, where it takes the same value as~$y$. As~$x_{i+1}$ and~$x_i$ coincide on~$R_i$ which determines~$q$,~$f(x_{i+1})$ and~$f(x_i)$ coincide at~$q$. As~$x_n$ coincides with~$y$ on~$R$ which determines~$q$,~$f(x_n)$ and~$f(y)$ coincide at~$q$. As a result,~$f(x)$ and~$f(y)$ coincide at~$q$, so~$\W_f(q)$ indeed determines~$q$.
\end{proof}
One should think of each output position~$q\in F$ as observing the input of~$f$ through its input window~$\W_f(q)$, and then choosing its content according to what it sees. A function~$f$ is~$r$-narrow if~$|\W_f(q)|\leq r$ for all~$q$.
%sss
\subsubsection{Relevance of the definition} We now discuss why a compact set in~$\nar$ always admits a generation procedure which is in a sense local.

Let~$X\in \nar$, i.e.~$X=\bigcup_{n\in\N} X_n$ where~$X_n$ is the image of some~$r_n$-narrow function~$f_n:A_n^\N\to \Sigma^{\Z^d}$. The procedure producing a configuration~$x\in X$ runs as follows:
\begin{itemize}
\item Choose a number~$n\in\N$ and an arbitrary sequence of symbols~$a\in A^\N_n$,
\item For each cell~$p\in\Z^d$, read the values of~$a$ in at most~$r_n$ positions, determine the value of $f_n(a)$ at position~$p$, and fill~$p$ with that value.
\end{itemize}
Therefore, each cell has its own procedure, which takes a bounded amount of input information and outputs its content, and all the cells run their procedures in parallel with no explicit communication. Some bounded amount of communication is implicitly made possible as the cells share limited pieces of input information, via the number~$n$ and the overlaps between the windows of the cells.

Again we do not address the computability or complexity of computing the finite input window of each position and applying the rule that determines the output, and our results only focus on the ``combinatorial'' complexity of subshifts. In practice, the class~$\nar$ will mainly be used to prove that particular subshifts do not belong to that class, which a fortiori prevents the existence of efficiently computable windows and rules. Conversely, for every particular subshift~$X\in\nar$ considered in this article, the windows and rules will be explicit and efficiently computable.

%SSSSSS
\section{Symbolic dynamics}\label{sec_dynamics}

We recall the classical vocabulary associated to subshifts.

%sss
\subsection{Definitions}
Let~$d\geq 1$ and~$\Sigma$ a finite alphabet.

We endow~$\Z^d$ with the metric~$d(p,q)=\max_i|p_i-q_i|$, where~$p=(p_1,\ldots,p_d)$ and~$q=(q_1,\ldots,q_d)$. For~$F\subseteq\Z^d$,~$\diam(F)=\max_{p,q\in F}d(p,q)$, which is~$\infty$ if~$F$ is infinite. For~$F,G\subseteq\Z^d$,~$d(F,G)=\min_{p\in F,q\in G}d(p,q)$. We will often consider the \textbf{cubes}~$S_n=[0,n-1]^d$ and~$Q_n=[-n,n]^d$, for~$n\in\N$.

Let~$F\subseteq\Z^d$ be any set. An~$F$\textbf{-pattern} is an element~$\pi\in \Sigma^F$. A \textbf{pattern} is an~$F$-pattern for some~$F$, which is called the \textbf{domain} of~$\pi$ and is denoted by~$\dom(\pi)$. A pattern~$\pi'$ \textbf{extends} a pattern~$\pi$ if~$\dom(\pi)\subseteq\dom(\pi')$ and~$\pi'(p)=\pi(p)$ for all~$p\in\dom(\pi)$. A \textbf{finite pattern} is an~$F$-pattern for some finite~$F\subseteq\Z^d$. If~$\pi,\pi'$ are patterns with disjoint domains~$F,F'$ respectively, then~$\pi\cup\pi'$ is the pattern with domain~$F\cup F'$ extending~$\pi$ and~$\pi'$. A \textbf{configuration} is an element~$x\in\Sigma^{\Z^d}$.

If~$\pi$ is a pattern, then~$[\pi]$ is the set of configurations extending~$\pi$. The set~$\Sigma^{\Z^d}$ is endowed with the topology generated by the sets~$[\pi]$ where~$\pi$ is a finite pattern. As already mentioned, the space~$\Sigma^{\Z^d}$ is endowed with the \textbf{shift action}, which is the continuous~$\Z^d$-action~$p\cdot x=y$ where~$y(q)=x(p+q)$. We will often write~$\sigma^p(x):=p\cdot x$. The shift also acts on patterns: if~$\pi$ is an~$F$-pattern, then~$\sigma^p(\pi)$ is the~$F'$-pattern~$\pi'$ where~$F'=F-p$ and~$\pi'(q)=\pi(p+q)$ for~$q\in F'$. If~$p\in\Z^d$ and~$x$ is a configuration, then we say that a pattern~$\pi$ \textbf{appears at position~$p$ in~$x$} if~$\sigma^p(x)$ extends~$\pi$.

A \textbf{subshift} is a closed subset~$X$ of~$\Sigma^{\Z^d}$ which is \textbf{shift-invariant}, i.e.~satisfies~$\sigma^p(X)=X$ for all~$p\in\Z^d$. We fix a subshift~$X\subseteq \Sigma^{\Z^d}$. All the subsequent notions are relative to~$X$, but we do not mention~$X$ which will always be clear from the context. A \textbf{valid configuration} is an element~$x\in X$. A \textbf{valid pattern} is a pattern appearing in some valid configuration, at any position or equivalently at the origin. Two disjoint regions~$F,G\subseteq\Z^d$ are \textbf{independent} if for every valid~$F$-pattern~$\pi$ and every valid~$G$-pattern~$\pi'$,~$\pi\cup \pi'$ is valid.

A \textbf{subshift of finite type (SFT)} is the subshift induced by a finite set~$P$ of finite patterns, called forbidden patterns, and defined as the set of configurations in which no~$\pi\in P$ appears. Let~$C$ be a finite set of colors. A \textbf{Wang tile} over~$C$ associates to each edge of the unit square a color in~$C$. A \textbf{Wang tileset} over~$C$ is a set~$T$ of Wang tiles over~$C$. It induces a~$\Z^2$-subshift of finite type~$X_T\subseteq T^{\Z^2}$ which is the set of configurations in which neighbor cells have the same color on their common edge.

%sss
\subsection{Dynamical properties}

We recall classical dynamical properties of subshifts, that can be found in \cite{Furstenberg67} or \cite{Vries13}.
\begin{definition}
A subshift~$X\subseteq\Sigma^{\Z^d}$ is \textbf{transitive} if for every pair of valid finite patterns~$\pi,\pi'$, there exists~$p\in\Z^d$ such that~$\pi\cup \sigma^p(\pi')$ is a valid pattern; in other words,~$\pi$ and~$\pi'$ appear in a common configuration. 

A subshift~$X\subseteq\Sigma^{\Z^d}$ is \textbf{strongly irreducible} if there exists~$n\in\N$ such that all regions~$F,G\subseteq \Z^d$ satisfying~$d(F,G)\geq n$ are independent.

A subshift~$X\subseteq\Sigma^{\Z^d}$ is \textbf{mixing} if for every~$m\in\N$ there exists~$n\in\N$ such that all regions~$F,G\subseteq \Z^d$ satisfying~$\diam(F)\leq m,\diam(G)\leq m$ and~$d(F,G)\geq n$ are independent.

A subshift~$X\subseteq\Sigma^{\Z^d}$ is \textbf{weakly mixing} if~$X\times X$ is transitive.
\end{definition}

Transitive systems are called ergodic in \cite{Furstenberg67} and are also referred as irreducible in the literature. In the definitions of transitivity, weak mixing and mixing, one can equivalently replace finite regions by cubes~$Q_m$. However, the version of strong irreducibility restricted to cubes is strictly weaker in general and is called block-gluing. One has the following chain of implications:
\begin{center}
strongly irreducible $\implies$ mixing $\implies$ weakly mixing $\implies$ transitive.
\end{center}
Strong irreducibility is a simple necessary condition for a subshift to be a factor of a fullshift.

The notion of weak mixing will be important in this article. Its meaning might not be clear at first, but the characterization given condition (3) in the next proposition is much more intuitive, easier to work with and is more apparently a relaxation of the mixing property. A part of this result was proved by Furstenberg \cite{Furstenberg67} in the context of continuous~$\Z$-actions.
\begin{proposition}[\ti{Characterizations of weak mixing}]
For a subshift~$X\subseteq \Sigma^{\Z^d}$, the following conditions are equivalent:
\begin{enumerate}
\item $X$ is weakly mixing,
\item For each~$k\geq 1$,~$X^k$ is transitive,
\item For every~$n\in\N$, there exists~$p\in\Z^d$ such that~$Q_n$ and~$p+Q_n$ are independent,
\item For every~$n\in\N$ and every valid~$Q_n$-pattern~$\pi$, there exists~$p\in\Z^d$ such that for every~$Q_n$-pattern~$\pi'$, $\pi\cup\sigma^p(\pi')$ is valid.%there exists~$x\in X$ such that~$\pi$ and~$\pi'$ appear in~$x$ at positions~$0$ and~$p$ respectively.
\end{enumerate}
\end{proposition}
The result still holds if one replaces the quantification over~$n$ by a quantification over the finite regions~$F\subseteq\Z^d$ (and the~$Q_n$-patterns by~$F$-patterns). \begin{proof}
$1.\Rightarrow 2.$ Let~$n\in\N$ and let~$\pi_1,\ldots,\pi_k$ and~$\pi'_1,\ldots,\pi'_k$ be valid~$Q_n$-patterns. Our goal is to prove the existence of~$p\in\Z^d$ such that each~$\pi_i\cup\sigma^p(\pi'_i)$ is valid. By induction on~$i\leq k$, we show the existence of~$p_1,\ldots,p_i$ such that~$\mu_i=\bigcup_{j\leq i}\sigma^{p_j}(\pi_j)$ and~$\mu'_j=\bigcup_{j\leq i}\sigma^{p_j}(\pi'_j)$ are both valid. Start with~$p_1=0$ and note that~$\pi_1$ and~$\pi'_1$ are valid by assumption. For~$i<k$, assume that~$p_1,\ldots,p_{i}$ have been defined. As~$X\times X$ is transitive, there exists~$p_{i+1}\in\Z^d$ such that~$\mu_i\cup\sigma^{p_{i+1}}(\pi_{i+1})$ and~$\mu'_i\cup\sigma^{p_{i+1}}(\pi'_{i+1})$ are valid, which proves the induction step. Finally, as~$X$ is transitive, there exists~$p$ such that~$\mu_k\cup \sigma^p(\mu'_k)$ is valid. Observe that this pattern contains each~$\sigma^{p_i}(\pi_i\cup\sigma^p(\pi'_i))$, so each~$p_i\cup\sigma^p(\pi'_i)$ is valid.

$2.\Rightarrow 3.$ Let~$n\in\N$. Let~$\{\pi_1,\ldots,\pi_m\}$ be the set of valid~$Q_n$-patterns and~$k=m^2$. The two following patterns are valid in~$X^k$:
\begin{align*}
\pi&=(\pi_1,\ldots,\pi_1,\pi_2,\ldots,\pi_2,\ldots,\pi_m,\ldots,\pi_m)\\
\pi'&=(\pi_1,\ldots,\pi_m,\pi_1,\ldots,\pi_m,\ldots,\pi_1,\ldots,\pi_m).
\end{align*}
As~$X^k$ is transitive, there exists~$p\in\Z^d$ such that~$\pi\cup\sigma^p(\pi')$ is valid in~$X^k$. It implies that for all~$i,j\leq m$,~$\pi_i\cup\sigma^p(\pi_j)$ is valid in~$X$. As a result,~$Q_n$ and~$p+Q_n$ are independent.

$3.\Rightarrow 4.$ Observe that condition~$3.$ is obtained from condition~$4.$ by exchanging the quantifiers over~$p$ and~$\pi$, so~$3.$ is indeed stronger than~$4.$ Let us prove it more precisely. If~$n\in\N$ and~$\pi$ is a valid~$Q_n$-pattern, then let~$p\in\Z^d$ be given by~$3.$ applied to~$n$. As~$Q_n$ and~$p+Q_n$ are independent, for every~$Q_n$-pattern~$\pi'$,~$\pi\cup\sigma^p(\pi')$ is valid.%If~$F,F'\subseteq\Z^d$ are finite and~$\pi$ is an~$F$-pattern, then let~$p\in\Z^d$ be given by~$2.$ applied to~$F,F'$. As~$F$ and~$p+F'$ are independent, for every~$F'$-pattern~$\pi'$ there exists~$x\in X$ in which~$\pi$ and~$\pi'$ appear at positions~$0$ and~$p$ respectively.

$4.\Rightarrow 1.$ Let~$n\in\N$ and~$\pi_1,\pi'_1,\pi_2,\pi'_2$ be valid~$Q_n$-patterns. Our goal is to find~$q\in\Z^d$ such that both~$\pi_1\cup \sigma^q(\pi'_1)$ and~$\pi_2\cup\sigma^q(\pi'_2)$ are valid patterns. Applying~$4.$ to~$n$ and~$\pi_1$, there exists~$p\in\Z^d$ such that~$\pi_1\cup\sigma^p(\pi_2)$ is valid. % and~$x\in X$ in which~$\pi_1$ and~$\pi_2$ appear at positions~$0$ and~$p$ respectively.
Let~$m\in\N$ be such that~$Q_m$ contains~$Q_n\cup (p+Q_n)$ and~$\pi$ be a valid~$Q_m$-pattern extending~$\pi_1\cup\sigma^p(\pi_2)$. %the~$Q_m$-pattern in~$x$.
Applying again~$4.$ to~$m$ and~$\pi$ gives some~$q\in\Z^d$. Let~$\pi''_1$ and~$\pi''_2$ be valid extensions on~$Q_m$ of~$\pi'_1$ and~$\sigma^p(\pi'_2)$ respectively. By choice of~$q$,~$\pi\cup\sigma^q(\pi''_1)$ and~$\pi\cup\sigma^q(\pi''_2)$ are valid. %there exists~$x_1\in X$ in which~$\pi$ and~$\pi''_1$ appear at positions~$0$ and~$q$ respectively, and there exists~$x_2\in X$ in which~$\pi$ and~$\pi''_2$ appear at positions~$0$ and~$q$ respectively. 
Finally observe that~$\pi\cup\sigma^q(\pi''_1)$ extends~$\pi_1\cup\sigma^q(\pi'_1)$ and~$\pi\cup\sigma^q(\pi''_2)$ extends~$\sigma^p(\pi_2\cup \sigma^q(\pi'_2))$. As a result,~$\pi_1\cup\sigma^q(\pi'_1)$ and~$\pi_2\cup\sigma^q(\pi'_2)$ are valid.
%Observe that~$\pi_1$ and~$\pi'_1$ appear in~$x_1$ at positions~$0$ and~$q$ respectively, and~$\pi_2$ and~$\pi'_2$ appear in~$\Sigma^p(x_2)$ at positions~$0$ and~$q$ respectively.
\end{proof}

In order to prove mixing properties, it is often convenient to consider regions having certain prescribed shapes.
\begin{proposition}
Let~$\F$ be a family of finite subsets of~$\Z^d$ such that every~$Q_m$ is contained, up to translation, in some element of~$\F$.

A~$\Z^d$-subshift is mixing if and only if for every~$F\in\F$ there exists~$n\in\N$ such that for every~$p\in\Z^d$ satisfying~$|p|\geq n$,~$F$ and~$p+F$ are independent.
\end{proposition}
\begin{proof}
Assume the property, let~$m\in\N$ and let~$H\in\F$ contain~$h+Q_m$ for some~$h\in \Z^d$. Applying the property to~$H$ yields some~$n$. Let~$F,G$ have diameters at most~$m$ and satisfy~$d(F,G)\geq n$. Choose any~$p\in F,q\in G$ and note that~$d(p,q)\geq n$. One has~$F\subseteq p+Q_m\subseteq q-h+H$ and~$G\subseteq q+Q_m\subseteq q-h+H$. As~$|(p-h)-(q-h)|=d(p,q)\geq n$,~$p-h+H$ and~$q-h+H$ are independent, so~$F$ and~$G$ are independent as well.

Conversely, assume that~$X$ is mixing, let~$F\in \F$ and let~$m=\diam(F)$. The mixing property yields some~$n$. Let~$n'=n+m$ and~$|p|\geq n'$. The regions~$F$ and~$p+F$ have diameter~$m$ and~$d(F,p+F)\geq |p|-\diam(F)\geq n$, so~$F$ and~$p+F$ are independent.
\end{proof}

%ssss
\subsection{Weak factors and weak conjugacies}
Usually, two~$\Z^d$-subshifts~$X$ and~$Y$ are considered equivalent when they are \emph{conjugate}, which means that there exists a homeomorphism~$f:X\to Y$ that commutes with the shift actions:~$\sigma^u\circ f=f\circ \sigma^u$ for all~$u\in\Z^d$.

We will consider more flexible notions of factor and conjugacy, which allow one to move the group elements by a homomorphism, and are the natural notions of homomorphisms and isomorphisms in the category of continuous group actions, when the group is not fixed.

The main interest of these notions is that the classes~$\gen$ and~$\nar$ are preserved by taking weak factors, as we will show in the next section.

\begin{definition}\label{def_equivalence}
Let~$X$ be a~$\Z^d$-subshift and~$Y$ a~$\Z^e$-subshift. We say that~$Y$ is a \textbf{weak factor} of~$X$ if there exists a continuous surjective function~$f:X\to Y$ and a homomorphism~$\varphi:\Z^e\to\Z^d$ such that
\begin{equation*}
\sigma^p\circ f=f\circ \sigma^{\varphi(p)}
\end{equation*}
for all~$p\in\Z^e$.

Two subshifts~$X,Y$ are \textbf{weakly conjugate} if moreover~$f$ is a homeomorphism and~$\varphi$ an isomorphism (which is possible only when~$d=e$).
\end{definition}

%\begin{remark}
%One could be tempted to consider an alternative definition of a weak factor, requiring a homomorphism~$\psi:\Z^d\to\Z^e$ satisfying~$\sigma^{\psi(u)}\circ F=F\circ \sigma^{u}$. However, the action on~$Y$ would not be uniquely determined from the action on~$X$ unless~$\psi$ is surjective; and if one assumes that~$\psi$ is surjective then, as~$\Z^e$ is a free group, it has a right-inverse~$\varphi$ which satisfies the definition given above:~$\sigma^u\circ F=\sigma^{\psi\circ \varphi(u)}\circ F=F\circ \sigma^{\varphi(u)}$. Therefore, the corrected version of the alternative definition would be more restrictive.
%\end{remark}

Note that a weak factor (resp.~weak conjugacy) map~$f:X\to Y$ can be seen as a factor (resp.~conjugacy) map~$f':X'\to Y$ where~$X'$ is~$X$ endowed with the~$\Z^e$-action~$p\star x:=\varphi(p)\cdot x$.

The Curtis-Hedlund-Lyndon theorem \cite[Theorem 6.2.9]{LM95}, stating that factor maps are equivalent to sliding block codes, immediately extends to weak factor maps.
\begin{proposition}[\ti{Curtis-Hedlund-Lyndon theorem}]\label{prop_hedlund}
Let~$\varphi:\Z^e\to\Z^d$ be a group homomorphism and~$f:A^{\Z^d}\to B^{\Z^e}$ be a function. The following conditions are equivalent:
\begin{itemize}
\item $f$ is continuous and satisfies~$\sigma^p\circ f=f\circ \sigma^{\varphi(p)}$,
\item There exists~$N\geq 0$ and a function~$\rho:Q_N\to B$, where~$Q_N=[-N,N]^d$, such that~$f(x)_p=\rho(x|_{\varphi(p)+Q_N})$.
\end{itemize}
\end{proposition}
We will call the function~$\rho$ a \textbf{rule}. When~$d=e$ and~$\varphi$ is the identity, we are back to the notion of sliding block code.
\begin{proof}
The usual proof when~$\varphi$ is the identity extends immediately as follows.

Assume the first conditions. As~$f$ is continuous, there exist~$N$ and~$\rho$ such that~$f(x)_0=\rho(x|_{Q_N})$. One has~$f(x)_p=(\sigma^p\circ f(x))_0=(f(\sigma^{\varphi(p)}(x))_0=\rho(\sigma^{\varphi(p)}(x)|_{Q_N})=\rho(x|_{\varphi(p)+Q_N})$.

Conversely, if~$f$ is given by~$N$ and~$\rho$, then~$f$ is continuous and~$(\sigma^q\circ f(x))_p=f(x)_{p+q}=\rho(x|_{\varphi(p)+\varphi(q)+Q_N})=\rho(\sigma^{\varphi(q)}(x)|_{\varphi(p)+Q_N})=f(\sigma^{\varphi(q)}(x))_p$, so~$\sigma^q\circ f=f\circ \sigma^{\varphi(q)}$.
\end{proof}

Let us illustrate the notions of weak factor and weak conjugacy on several examples.

\begin{example}[\ti{Fullshifts of different dimensions}]
If~$d>e$, then the fullshift~$X_e:=\Sigma^{\Z^e}$ is a weak factor of the fullshift~$X_d:=\Sigma^{\Z^d}$. Let~$\varphi:\Z^e\to \Z^d$ be any injective homomorphism and~$f:X_d\to X_e$ be defined by~$f(x)=x\circ \varphi$. One can check that~$\sigma^p\circ f=f\circ \sigma^{\varphi(p)}$ and~$f$ is surjective because for every~$y\in X_e$, one can build a pre-image~$x\in X_d$ by filling the positions in~$\im{\varphi}$ (without any conflict as~$\varphi$ is injective) and filling the rest arbitrarily.

However,~$X_d$ is not a weak factor of~$X_e$. Indeed, if~$f:X_e\to X_d$ is a weak factor map via a homomorphism~$\varphi:\Z^d\to\Z^e$, then~$\varphi$ is non-injective and if~$p\in\ker\varphi$ then~$f(X_e)$ only contains~$p$-periodic configurations, as~$\sigma^p\circ f=F\circ \sigma^{\varphi(p)}=f$.
\end{example}

\begin{example}[\ti{Geometric transformation}]\label{ex_geom}
Let~$\varphi:\Z^d\to\Z^d$ be an isomorphism. It induces a homeomorphism~$f:\Sigma^{\Z^d}\to\Sigma^{\Z^d}$ defined by~$f(x)=x\circ \varphi$. If~$X\subseteq \Sigma^{\Z^d}$ is a subshift, then~$f(X)$ is also a subshift. It is usually not conjugate to~$X$, but it is weakly conjugate because again,~$\sigma^p\circ f=f\circ \sigma^{\varphi(p)}$ for all~$p\in\Z^d$.

For instance, a~$\Z^2$-subshift can be reflected accross an axis, or rotated by~$90$ degrees, yielding aweakly conjugate subshift.
\end{example}

\begin{example}[\ti{Higher power presentation}]\label{ex_block}
Let~$X\subseteq \Sigma^{\Z^d}$ be a subshift,~$a=(a_1,\ldots,a_d)\in\Z^d$ have positive coordinates and~$R_a=[0,a_1-1]\times \ldots\times[0,a_d-1]$. The~$a$-higher power presentation of~$X$ is the subshift~$[X]_a\subseteq \Gamma^{\Z^d}$, where~$\Gamma=\Sigma^{R_a}$, obtained by dividing~$\Z^d$ into a regular grid of copies of the rectangle~$R_a$ and seeing each pattern in such a rectangle as a symbol in the alphabet~$\Gamma$ (it can be found in \cite[Definition 1.4.4]{LM95} in the one-dimensional case). More formally,~$[X]_a$ is the image of~$X$ by the function~$f:\Sigma^{\Z^d}\to\Gamma^{\Z^d}$ defined by~$f(x)=y$ where~$y(p)$ is the~$R_a$-pattern appearing at position~$ap=(a_1p_1,\ldots,a_dp_d)$ in~$x$.

One has~$\sigma^p\circ f=f\circ \sigma^{ap}$, so~$[X]_a$ is a weak factor of~$X$ via the homomorphism~$\varphi(p)=ap$.% (one can also see~$f$ as a conjugacy from~$X$ endowed with the action~$p\cdot x=\sigma^{ap}(x)$ to the subshift~$[X]_p$).
\end{example}

\begin{example}[\ti{Fullshifts on different alphabets}]
Let~$X=\Sigma^{\Z^d}$ and~$Y=\Gamma^{\Z^d}$ be~$\Z^d$-fullshifts, with~$|\Sigma|<|\Gamma|$. First,~$X$ is a factor of~$Y$: fix some surjective function~$\Gamma\to\Sigma$ and extend it from~$Y$ to~$X$. Conversely,~$Y$ is not a factor of~$X$ because the entropy of~$Y$ is~$h(Y)=\log |\Gamma|$ which is strictly larger than the entropy of~$X$,~$h(X)=\log|\Sigma|$ (see \cite{LM95}). However,~$Y$ is a weak factor of~$X$. Indeed, let~$n\in\N$ be such that~$|\Sigma^{n^d}|\geq |\Gamma|$ and~$a=(n,\ldots,n)\in\Z^d$. The~$a$-higher power presentation~$[X]_a$ is the full-shift on the alphabet~$\Sigma^{n^d}$ so~$Y$ is a factor of~$[X]_a$, which is a weak factor of~$X$ (see Example \ref{ex_block}), so~$Y$ is a weak factor of~$X$.
\end{example}
%
%\begin{example}[\ti{Periodic subshifts}]
%Let~$d<e$,~$X$ be a~$\Z^d$-subshift and~$Y$ a~$\Z^e$-subshift. If~$Y$ is a weak factor of~$X$, then~$Y$ has a period, i.e.~there exists a non-zero vector~$u\in\Z^e$ such that~$\sigma^u(y)=y$ for all~$y\in Y$. Indeed, let~$F:X\to Y$ and~$\varphi:\Z^e\to\Z^d$ witness that~$Y$ is a weak factor of~$X$. As~$d<e$,~$\varphi$ is not injective, and for any~$u\in\ker\varphi$,~$\sigma^u\circ F=F\circ \sigma^{\varphi(u)}=F$.
%\end{example}

Contrary to factors, weak factors do not preserve most of the dynamical properties, except the mixing property in some cases, as stated in the next result.
\begin{proposition}\label{prop_weak_factor_mixing}
Let~$X$ be a mixing subshift. If a subshift~$Y$ is a weak factor of~$X$, then~$Y$ satisfies at least one of the two following properties:
\begin{itemize}
\item $Y$ is periodic in the sense that there exists~$p\neq 0$ such that~$\sigma^p(y)=y$ for all~$y\in Y$,
\item Or~$Y$ is mixing.
\end{itemize}
\end{proposition}
\begin{proof}
Let~$f:X\to Y$ be a weak factor map via~$\varphi:\Z^e\to\Z^d$. If~$\varphi$ is not injective, then let~$p\in\ker(\varphi)$ be non-zero. One has~$\sigma^p\circ f=f\circ \sigma^{\varphi(p)}=f$, so~$\sigma^p(y)=y$ for all~$y\in Y$.

Now assume that~$\varphi$ is injective. Let~$N\geq 0$ and~$\rho$ be given by Proposition \ref{prop_hedlund}, satisfying~$F(x)_p=\rho(x|_{\varphi(q)+Q_N})$. Let~$m\in\N$, let~$m'=\diam\varphi([0,m]^e)+2N$, let~$n'$ be associated to~$m'$ using the mixing property of~$X$, and~$n$ be such that if~$|p|\geq n$, then~$\varphi(p)\geq n'+2N$. Such an~$n$ exists as~$\varphi$ is injective. Let~$F,G\subseteq\Z^e$ be two regions of diameters at most~$m$, and at distance at least~$n$ from each other. One has~$\W_f(F)\subseteq \varphi(F)+Q_N$ and~$\W_f(G)\subseteq \varphi(G)+Q_N$, so~$\W_f(F)$ and~$\W_f(G)$ have diameters at most~$m'$. Moreover,~$d(F,G)\geq n$ implies that~$d(\varphi(F),\varphi(G))\geq n'+2N$, so~$d(\W_f(F),\W_f(G))\geq n'$. By the mixing property of~$X$,~$\W_f(F)$ and~$\W_f(G)$ are independent, so~$F$ and~$G$ are independent as well.
\end{proof}

%SSSSSSSSS
\section{Properties of the classes \texorpdfstring{$\gen$ and $\nar$}{L0,L2}}\label{sec_properties}

We investigate general properties of the classes.

First, belonging to~$\gen$ implies the existence of strongly periodic configurations. A configuration~$x\in\Sigma^{\Z^d}$ is \textbf{strongly periodic} if its orbit~$\{\sigma^p(x):p\in\Z^d\}$ is finite, equivalently if there exists~$k\in\N\setminus \{0\}$ such that~$\sigma^{kp}(x)=x$ for all~$p\in\Z^d$. It is known that every countable subshift contains a strongly periodic configuration \cite{Dolbilin95,BallierDJ08} (indeed, any configuration of maximal Cantor-Bendixson rank is strongly periodic).
\begin{proposition}
Every subshift in~$\gen$ has strongly periodic configurations.
\end{proposition}
\begin{proof}
Let~$f:K\times A^E\to X$ be a factor map. As~$K$ is countable, it contains a strongly periodic configuration~$x$. %there exists a minimal countable ordinal~$\alpha$ such that the iterated Cantor-Bendixson derivative~$K^{(\alpha+1)}$ is empty, hence~$K^{(\alpha)}$ is finite, because it is compact and only contains isolated points \cite[Section 6.C]{Kechris95}. Let~$(e_i)_{i<d}$ be the canonical basis of~$\Z^d$. Each~$e_i$ acts as a homeomorphism on~$K$ so it sends~$K^{(\alpha)}$ to itself. As that set is finite, each one of its elements is periodic under the action of~$e_i$. Choose some~$x\in K^{(\alpha)}$ and let~$k$ be a common period of~$x$ under the actions of the vectors~$e_i$. One has~$\sigma^{kp}(x)=x$ for all~$p\in\Z^d$.
Let~$a$ be a constant sequence in~$A^e$. The orbit of~$(x,a)$ is finite, so the orbit of~$y=f(x,a)$ is finite as well.% As~$\sigma^{kp}(a)=a$ and~$f$ commutes with the shift action, one has~$\sigma^{kp}(y)=y$ for all~$p\in\Z^d$.
\end{proof}

The next result investigates simple closure properties of narrow functions. It will only be used in Section \ref{sec_rotation}, and can be skipped at first reading.
\begin{proposition}\label{prop_narrow}
Let~$\Sigma$ be a finite alphabet,~$d\geq 1$ and~$X,Y\subseteq\Sigma^{\Z^d}$ be compact.
\begin{itemize}
\item If~$X,Y$ are images of narrow functions, then so is~$X\cup Y$,
\item If~$X$ is the image of a narrow function and~$\pi$ is a finite pattern, then~$X\cap [\pi]$ is the image of a narrow function,
\item If~$X\in\nar$, then there exists a finite pattern~$\pi$ such that~$X\cap [\pi]$ is the image of a narrow function.
\end{itemize}
\end{proposition}
\begin{proof}
Let~$f:A^\N\to X$ and~$g:A^\N\to Y$ be surjective narrow functions with~$|A|\geq 2$ (we can indeed assume that they start from the same alphabet). Let~$a_0\in A$ and note that~$A^\N\cong A\times A^\N$. Define~$h:A\times A^\N\to X\cup Y$ sending~$(a,s)$ to~$f(s)$ if~$a=a_0$, to~$g(s)$ otherwise. For any~$p\in\Z^d$,~$|\W_f(p)|=1+|\W_{f_X}(p)|+|\W_{f_Y}(p)|$ which is bounded.

Assume that~$X$ is the image of a narrow function~$f:A^\N\to X$ and let~$\pi$ be a finite pattern with domain~$F$. Let~$\xi$ be a pattern on the input window~$\W_f(F)$ which is sent to~$\pi$ by~$f$. Let~$f_\xi:A^\N\to X$ be defined as follows: given~$a$, let~$a'$ be obtained by replacing the content of~$\W_f(F)$ by~$\xi$, and let~$f_\xi(a)=f(a')$. $f_\xi$ is a narrow function. One has~$X=\im{f}=\bigcup_\xi\im{f_\xi}$. As there are finitely many possible patterns~$\xi$,~$X$ is a finite union of images of narrow functions, so it is itself the image of a narrow function by the first assertion.

By the Baire category theorem, if~$X=\bigcup_{i\in\N}X_i$ where each~$X_i$ is an image of a narrow function, then there exists a finite pattern~$\pi$ and~$i$ such that~$X\cap [\pi]\subseteq X_i$. As~$X_i$ is the image of a narrow function, so is~$X_i\cap [\pi]$ by the second assertion.
\end{proof}
%sss
\subsection{Closure by weak factors}
We now show that the classes~$\gen$ and~$\nar$ are closed under weak factors.

It is easier to prove this result by first showing that one can relax the definition of~$\gen$ by replacing countable~$\Z^d$-subshifts by any~$\Z^d$-space which is countable, compact and zero-dimensional, giving an equivalent definition. %It will be often convenient to show that a subshift belongs to~$\gen$ by using such countable~$\Z^d$-spaces, without having to convert them into subshifts.

\begin{proposition}\label{prop_countable_Zd_space}
Let~$K$ be a countable compact zero-dimensional~$\Z^d$-space and~$E$ a countable~$\Z^d$-set. If a~$\Z^d$-subshift~$X$ is a factor of~$K\times A^E$, then~$X\in\gen$.
\end{proposition}
\begin{proof}
We build an intermediate countable~$\Z^d$-subshift~$K'$ which is a factor of~$K$, and such that~$X$ is a factor of~$K'\times A^E$. The idea is to choose a finite clopen partition~$\Gamma$ of~$K$ and to build the corresponding symbolic dynamics: to each~$k\in K$ is associated the sequence of cells in~$\Gamma$ to which the iterates of~$k$ belong; the set of such sequences is a subshift.

Let~$\Sigma$ be the alphabet of~$X$. By compactness of~$K$ and~$A^E$, and zero-dimensionality of~$K$, there exists a finite set~$\Gamma$ and continuous maps~$g:K\to\Gamma$ and~$f:\Gamma\times A^E\to\Sigma$ such that~$F(k,a)_0=f(g(k),a)$ (the preimages of elements of~$\Gamma$ by~$g$ form a clopen partition of~$K$). Let~$G:K\to\Gamma^{\Z^d}$ be defined by~$G(k)_p=g(p\cdot k)$ and let~$K'=G(K)$. As~$K$ is countable, so is~$K'$. As~$K$ is compact and~$G$ is continuous,~$K'$ is compact. One easily checks that~$\sigma^p(G(k))=G(p\cdot k)$, so~$G$ is a factor map and~$K'$ is shift-invariant. Therefore,~$K'$ is a countable compact~$\Z^d$-subshift.

One has~$F(k,a)_p=f(G(k)_p,p\cdot a)$. Indeed,
\begin{align*}
F(k,a)_p&=(p\cdot F(k,a))_0\\
&=(F(p\cdot k,p\cdot a)_0\\
&=f(g(p\cdot k),p\cdot a)\\
&=f(G(k)_p,p\cdot a).
\end{align*}
Let~$H:K'\times A^E\to X$ be defined by~$H(k',a)_p=f(k'_p,p\cdot a)$. The previous equality shows that for any~$k\in K$ and~$a\in A^E$, one has~$F(k,a)=H(G(k),a)$. That~$H$ is a factor map easily follows from the fact that~$F$ and~$G$ are factor maps. Therefore,~$X$ is a factor of~$K'\times A^E$, so~$X\in\gen$.
\end{proof}

We then derive the announced result.
\begin{proposition}[\ti{$\gen$ is closed under weak factors}]
If~$Y$ is a weak factor of~$X$ and~$X\in \gen$, then~$Y\in \gen$.
\end{proposition}
\begin{proof}
Let~$K$ be a countable~$\Z^d$-subshift and~$E$ a countable~$\Z^d$-set. A weak factor of a factor of~$K\times A^E$ is a weak factor of~$K\times A^E$, so it is sufficient to show that any weak factor of~$K\times A^E$ belongs to~$\gen$.

Let then~$X\subseteq\Sigma^{\Z^e}$ a~$\Z^e$-subshift,~$\varphi:\Z^e\to\Z^d$ a group homomorphism and~$F:K\times A^E\to X$ a continuous surjective map such that~$\sigma^u\circ F(k,a)=F\circ \sigma^{\varphi(u)}(k,a)$.

We build a countable compact~$\Z^e$-space~$K'$ and a countable~$\Z^e$-set~$E'$ such that~$X$ is a factor of~$K'\times A^{E'}$. Let~$K'=K$ with the continuous~$\Z^e$-action~$p\star k=\varphi(p)\cdot k$ and~$E'=E$ with the~$\Z^e$-action~$p\star e=\varphi(p)\star e$. One has~$\sigma^u\circ F(k,a)=F(u\star (k,a))$ so~$F$ is indeed a factor map for the~$\star$ action. Note that~$K'$ is not necessarily a~$\Z^d$-subshift, but it is a countable compact zero-dimensional~$\Z^d$-space, so~$X\in\gen$ by Proposition \ref{prop_countable_Zd_space}.
\end{proof}

A similar result holds for~$\nar$.
\begin{proposition}[\ti{$\nar$ is closed under weak factors}]
Let~$Y$ is a weak factor of~$X$. If~$X$ is the image of a narrow function, then so is~$Y$. If~$X\in\nar$, then~$Y\in\nar$.
\end{proposition}
\begin{proof}
Every weak factor map is narrow, and a composition of narrow functions is narrow.
\end{proof}

\begin{example}[\ti{Building a~$\Z^2$-subshift from a~$\Z$-subshift}]
To a~$\Z$-subshift~$X\subseteq \Sigma^\Z$ we can associate a~$\Z^2$-subshift~$Y$, obtained by taking any configuration~$x\in X$ and copying it on each row. More precisely, let~$F:\Sigma^{\Z}\to\Sigma^{\Z^2}$ be defined by~$F(x)=y$ where~$y_{(i,j)}=x_i$. The function~$F:X\to Y$ is a weak factor map via the homomorphism~$\varphi:\Z^2\to\Z$ sending~$(i,j)$ to~$i$. Therefore, if~$X\in\mathscr{L}^i$ then~$Y\in\mathscr{L}^i$ as well.
\end{example}

%ssssssss
\subsection{Weak mixing and countable unions}
Most of the subshifts that we consider in this article are weakly mixing. It turns out that this dynamical property makes it significantly simpler to prove that a subshift does not belong to~$\gen$ or~$\nar$.

%sss
\subsubsection{Compact classes of functions}
We first state general results that we will apply to local functions (to be defined) and narrow functions.

If~$A$ is a finite alphabet and~$E$ is countably infinite, then~$A^E$ can be identified with~$A^\N$, endowed with the metric~$d(x,y)=2^{-n}$ where~$n$ is minimal such that~$x(n)\neq y(n)$.

Let~$A,B$ be finite alphabets and~$Y\subseteq A^\N$ be a compact set. The space of continuous functions~$f:Y\to B^\N$ is endowed with the \textbf{compact-open} topology, generated by the sets~$\{f:f[u]\subseteq [v]\}$ as a subbasis, where~$u\in A^*$ and~$v\in B^*$ are finite sequences and~$[u],[v]$ are the corresponding cylinders. For a function~$h:\N\to\N$, the class
\begin{equation*}
\K_h:=\{f:d(x,y)\leq 2^{-h(n)}\implies d(f(x),f(y))\leq 2^{-n}\}
\end{equation*}
is compact in this topology. The function~$h$ is called a modulus of uniform continuity for the functions in~$\K_h$, and expresses that the first~$n$ symbols of~$f(x)$ depend on the first~$h(n)$ symbols of~$x$.

Conversely, for every compact class~$\K$, there exists~$h:\N\to\N$ such that~$\K\subseteq \K_h$. All in all, a class~$\K$ is compact if and only if there exists a closed class~$\C$ and a function~$h$ such that~$\K=\K_h\cap\C$.

\begin{proposition}\label{prop_compact_class}
Let~$Y$ be compact,~$B$ be a finite alphabet and~$\K$ be a compact class of continuous functions from~$Y$ to~$B^\N$.

If a compact set~$X\subseteq B^\N$ is not the image of any~$f\in\K$, then there exists~$N$ such that~$X$ differs from the image of every~$f\in\K$ on~$[0,N]$.
\end{proposition}
\begin{proof}
For each~$f\in\K$, there exists~$n$ such that~$X$ differs from~$\im{f}$ on~$[0,n]$. The operator~$f\mapsto \im{f}$ is continuous, so for all~$g\in\K$ in a neighborhood of~$f$,~$X$ differs from~$\im{g}$ on~$[0,n]$. As~$\K$ is compact,~$\K$ is covered by finitely many such neighborhoods. Let~$N$ be the maximal~$n$ associated with these neighborhoods. For every~$f\in\K$,~$X$ differs from~$\im{f}$ on~$[0,N]$.
\end{proof}
When applying this result to~$B^{\Z^d}$, we can replace~$[0,N]$ with~$Q_N=[-N,N]^d\subseteq \Z^d$.

The next result shows that under certain conditions, allowing countable unions does not help when expressing a weakly mixing subshift.

\begin{proposition}[\ti{Weak mixing and countable unions}]\label{prop_sigma_compact}
Let~$\C=\bigcup_{i\in\N}\K_i$ where each~$\K_i$ is a compact class of functions from a compact space~$Y_i$ to~$B^{\Z^d}$ such that for all~$i\in\N$, all~$f\in\K_i$ and all~$p\in\Z^d$, there exists~$g\in\K_i$ such that~$\im{g}=\im{\sigma^p\circ f}$.

If~$X$ is a weakly mixing subshift, then~$X$ is a countable union of images of functions in~$\C$ if and only if~$X$ is the image of a single function in~$\C$.
\end{proposition}
\begin{proof}
Assume~$X$ is not the image of any function in~$\C$. Assume for a contradiction that~$X=\bigcup_{j\in\N} X_j$ where each~$X_j$ is the image of a function~$f_j\in\C$. By the Baire category theorem (Theorem \ref{thm_baire}), there exist~$j,k\in\N$ and a finite valid pattern~$\pi$ on~$Q_k$ such that~$X_j$ contains all the valid extensions of~$\pi$. Let~$i$ be such that~$f_j\in\K_i$. By Proposition \ref{prop_compact_class}, there exists~$n$ such that for every~$g\in\K_i$,~$X$ differs from~$\im{g}$ on~$Q_n$. As~$X$ is weakly mixing, there exists~$p\in\Z^2$ such that~$Q_k$ and~$p+Q_n$ are independent. In particular for every valid~$Q_n$-pattern~$\pi'$,~$\pi\cup\sigma^p(\pi')$ is valid. Any valid configuration extending~$\pi\cup\sigma^p(\pi')$ extends~$\pi$ so belongs to~$X_j$. By assumption about~$\C$, there exists~$g\in\K_i$ such that~$\im{g}=\im{\sigma^{-p}\circ f_j}$. It implies that~$X$ coincides with~$\im{g}$ on~$Q_n$, which contradicts the choice of~$n$.
\end{proof}

%\begin{remark}
%Observe that a topologically transitive/irreducible subshift is not a union of subshifts properly contained in it. So if it does not belong to a class of subshifts, then it is not a union of subshifts of that class. In particular, if it is not a factor of some~$A^E$, then it is not a union of factors of~$A^E$'s.
%\end{remark}
%
%sss
\subsubsection{Applications to local and narrow functions}
We recall that a subshift~$X$ belongs to~$\gen$ if there is a surjective factor map~$f:K\times A^E\to X$, where~$K$ is a countable~$\Z^d$-subshift and~$E$ a countable~$\Z^d$-set. We show that if~$X$ is weakly mixing, then one can get rid of~$K$ at the price of relaxing the property of~$f$ being a factor map. Let first define the relevant relaxation.

%The definition of~$\gen$ involves factor maps, which are cellular automata with possibly different input and output alphabets. We need to consider the following version of non-uniform cellular automata.
\begin{definition}[\ti{Local function}]
A function~$f:A^{\Z^d}\to B^{\Z^d}$ is a \textbf{local function of radius~$r$} if each~$p\in\Z^d$ is determined by~$\Nei(p,r)=\{q\in\Z^d:d(p,q)\leq r\}$.
\end{definition}
When~$A=B$, a local function is a non-uniform cellular automaton (Example \ref{ex_non-uniform_CA}).

We then state the announced result.

\begin{proposition}[\ti{Weak mixing vs countable unions}]\label{prop_wmix_local}
If a weakly mixing subshift belongs to~$\gen$, then it is the image of~$A^{\Z^d}\times \prod_{i=1}^kA^{\Z^d/H_i}$ under a local function, where~$A$ is a finite alphabet and each~$H_i$ is a non-trivial subgroup of rank at most~$d-1$.
\end{proposition}
\begin{proof}
If~$X\in\gen$, then~$X$ is the image of~$K\times Y$ with~$Y=A^{\Z^d}\times \prod_{i=1}^kA^{\Z^d/H_i}$ under a factor map~$f:K\times Y\to X$. For each~$k\in K$, the function~$f_k:Y\to X$ sending~$y$ to~$f(k,y)$ is a local function of a certain radius~$r$. The set~$\K$ of local functions~$g:Y\to \Sigma^{\Z^d}$ of radius~$r$ is a compact class of functions. Moreover, if~$g\in \K$ and~$p\in\Z^d$, then~$\sigma^p\circ g\circ\sigma^{-p}\in\K$ and its image is~$\im{\sigma^p\circ g}$. As~$X$ is weakly mixing and~$X=\bigcup_{k\in K}\im{f_k}$, we can apply Proposition \ref{prop_sigma_compact}, implying that~$X$ is the image of a single~$g\in\K$.
\end{proof}

An analogous result holds for the class~$\nar$: as far as weakly mixing subshifts are concerned, one can get rid of countable unions in the definition of~$\nar$.
\begin{proposition}[\ti{Weak mixing vs countable unions}]\label{prop_wmix_narrow}
A weakly mixing subshift belongs to~$\nar$ if and only if it is the image of a narrow function.
\end{proposition}

We need the following result.
\begin{lemma}\label{lem_narrow_compact}
Let~$A,B,d$ be fixed. For each~$r$, there exists a compact class~$\K$ of functions whose images are exactly the images of~$r$-narrow functions~$f:A^\N\to B^{\Z^d}$.
\end{lemma}
\begin{proof}
Let~$(p_n)_{n\in\N}$ be a one-to-one enumeration of~$\Z^d$. Let~$\K$ be the set of functions~$f:A^\N\to B^{\Z^d}$ such that for each~$i$,~$p_n$ is determined by a region of size~$r$ contained in~$[0,r(n+1))$. This class is compact and each function in~$\K$ is~$r$-narrow. Moreover, every~$r$-narrow function~$f$ has the same image as some~$g\in\K$, obtained as follows. The idea is that we tighten the input positions as close to~$0$ as possible. We first let~$W=\bigcup_{p\in\Z^d}\W_f(p)$ and define an injective function~$h:W\to\N$. We start defining an injective function~$h:\W_f(p_0)\to [0,r)$, which is possible as~$|\W_f(p_0)|\leq r$. Inductively assume that~$h:\W_f(p_0)\cup\ldots\cup \W_f(p_n)\to [0,r(n+1))$ has been defined. We extend~$h$ by sending~$\W_f(p_{n+1})\setminus (\W_f(p_0)\cup\ldots\cup \W_f(p_n))$ to~$[r(n+1),r(n+2))$ in an injective way, which is possible as~$|\W_f(p_{n+1})|\leq r$.

Given~$x\in A^\N$, let~$y\in A^\N$ be defined by~$y(k)=x(h(k))$ and let~$g(x)=f(y)$. Note that~$g(x)(p_n)$ is determined by the values of~$y$ on~$\W_f(p)$, i.e.~by the values of~$x$ on~$h(\W_f(p_n))$ which has size at most~$r$ and is contained in~$[0,r(n+1))$, so~$g$ belongs to~$\K$. One has~$\im{g}\subseteq \im{f}$ because each~$g(x)$ is~$f(y)$ for some~$y$. Conversely, one has~$\im{f}\subseteq\im{g}$ because for each~$y$, define~$x(k)=y(h^{-1}(k))$ if~$k\in\im{h}$ and let~$x(k)$ be an arbitrary symbol otherwise: one has~$g(x)=f(y)$.
\end{proof}

\begin{proof}[Proof of Proposition \ref{prop_wmix_narrow}]
By Lemma \ref{lem_narrow_compact}, the narrow functions can be arranged as a countable union of compact classes~$\bigcup_{A,r}\K_{A,r}$ satisfying the assumptions of Proposition \ref{prop_sigma_compact}.
\end{proof}

%sssssssss
\subsection{Separating the classes}
We show that~$\nar$ is strictly larger than~$\gen$ (in the class of subshifts).
\begin{proposition}\label{prop_separation}
There exists a subshift in~$\nar\setminus \gen$.
\end{proposition}
\begin{proof}

Let~$f:\{0,1\}^\N\to\{0,1\}^\Z$ send~$x=x_0x_1x_2\ldots$ to~$y=\ldots x_{2}^3x_{1}^2x_0^1x_1^2x_2^3\ldots$, where~$x_0$ is placed at position~$0$. Let~$X\subseteq \{0,1\}^\Z$ be the smallest subshift containing~$\im{f}$.

We first show that~$X\in\nar$. The function~$f$ is~$1$-narrow as each output position is determined by $1$ input position. $X$ is the union of the images of~$\sigma^k(\im{f})$ for~$k\in\Z$, and of the limit sequences~$\ldots000111\ldots$ and~$\ldots111000\ldots$. Therefore,~$X$ is a countable union of singletons and images of~$1$-narrow functions. Each singleton is the image of a~$0$-narrow function, so~$X\in\nar$.

Assume for a contradiction that~$X$ belongs to~$\gen$, i.e.~$X=\bigcup_{i\in\N}X_i$ where~$X_i$ is the image of a fullshift under a local function. Let~$\pi$ be the pattern~$010$ on~$[-1,1]$ and note that~$[\pi]\cap X=[\pi]\cap \im{f}$. By the Baire category theorem, there exist~$i\in\N$ and a valid finite pattern~$\pi'$ extending~$\pi$ such that~$[\pi']\cap X\subseteq X_i$. Let~$g:A^\Z\to X_i$ be a surjective local function, let~$W=[-k,k]$ its window. Extend~$\pi'$ to a valid pattern whose domain is~$[1-\frac{n(n+1)}{2},\frac{n(n+1)}{2}-1]$ for some~$n\geq 2k+2$. Let~$x=\ldots 000\pi'000\ldots$ and~$y=\ldots 111\pi'111\ldots$. One has~$x,y\in [\pi']\cap X\subseteq \im{g}$, let~$x_0,y_0\in A^\Z$ be pre-images by~$g$ of~$x,y$ respectively. Let~$z_0\in A^\Z$ coincide with~$x_0$ on~$(-\infty,\frac{n(n+1)}{2}+k]$ and with~$y_0$ on~$[\frac{(n+1)(n+2)}{2}-k,+\infty)$, which is possible as~$(\frac{(n+1)(n+2)}{2}-k)-(\frac{n(n+1)}{2}+k)=n+1-2k\geq 1$. Let~$z=g(z_0)\in [\pi']\cap X=[\pi']\cap \im{f}$. One has~$z(-\frac{(n+1)(n+2)}{2})=0$ and~$z(\frac{(n+1)(n+2)}{2})=1$, which contradicts the fact that~$z\in\im{f}$. We have reached a contradiction, so~$X\notin\gen$.
\end{proof}

%SSSSSS
\section{Ramified subshifts}\label{sec_obstructions}
This section contains one of the main results of the article, which is a technique that can be used to show that a two-dimensional subshift does not belong to~$\gen$. The idea is to identify an invariant satisfied by finite products of basic subshifts and preserved under factor maps.

%ssss
\subsection{Definition of ramifications}
We first introduce a definition which expresses the possibility of building a valid configuration from another one by inserting a pattern at a position and allowing modifications in a small neighborhood of the pattern.

If~$F\subseteq\Z^2$ and~$r\in\N$, then let~$\Nei(F,r)=\{p:\exists q\in F,d(p,q)\leq r\}$ be the~$r$-neighborhood of~$F$.
\begin{definition}[\ti{Graft}]
Let~$X$ be a~$\Z^2$-subshift,~$F\subseteq\Z^2$,~$\pi$ an~$F$-pattern,~$x\in X$ and~$r\in\N$. We say that~$\pi$ can be \textbf{$r$-grafted into~$x$ at position~$p\in\Z^2$} if there exists~$y\in X$ that coincides with~$x$ outside~$\Nei(p+F,r)$ and such that~$\pi$ appears at position~$p$ in~$y$, i.e.~$y(p+q)=\pi(q)$ for~$q\in F$.
\end{definition}
Observe that~$X$ is strongly irreducible if and only if there exists~$r$ such that for every finite region~$F$, every valid~$F$-pattern can be~$r$-grafted into every valid configuration (at any position, or equivalently at the origin).

The following condition about a configuration of a~$\Z^2$-subshift~$X$ roughly expresses that it contains patterns appearing at regularly located positions, which cannot be moved at certain other positions.

%\begin{itemize}
%\item For each~$r$ there exists an infinite set~$V_r\subseteq\Z^2$ of pairwise independent vectors and for each~$v\in V_r$ there exists an~$(r,v)$-obstruction in~$X$.
%%\item For each~$r$ there exists an infinite set~$V_r\subseteq\Z^2$ of pairwise independent vectors and for each~$v\in V_r$ there exists a valid configuration~$y$, a finite region~$F\subseteq\Z^2$ and a vector~$u$ independent of~$v$ such that for every~$\mu>0$ and every~$\lambda\in\Z$, the~$F$-pattern at position~$\lambda v+\mu u$ in~$y$ cannot be~$r$-grafted at position~$\lambda v$ in~$y$.%re is no valid configuration~$z$ that coincides with~$x$ outside~$N(F_\lambda,r)$ and with~$\Sigma^{nu}(y)$ on~$F_\lambda$.
%\end{itemize}

\begin{definition}[\ti{Ramification}]
Let~$X$ be a~$\Z^2$-subshift. A \textbf{ramification} is a configuration~$x\in X$ together with~$r\in\N,u,v\in\Z^2$ and~$F\subseteq\Z^2$ such that for all~$\lambda,\mu\in\Z$ with~$\mu>0$, the~$F$-pattern appearing at position~$\mu u+\lambda v$ in~$x$ cannot be~$r$-grafted into~$x$ at position~$\lambda v$. We say that~$x$ is an~$(r,v)$\textbf{-ramification}.
\end{definition}

Say that two vectors~$v_1,v_2\in\Z^2$ are \textbf{independent} if~$\lambda_ 1v_1+\lambda_2v_2=0$ implies~$\lambda_1=\lambda_2=0$, for~$\lambda_1,\lambda _2\in\Z$.
\begin{definition}[\ti{Ramified subshift}]
A~$\Z^2$-subshift~$X$ is \textbf{ramified} if for every~$r\in\N$ there exist infinitely many pairwise independent vectors~$v$ such that~$X$ admits a ramification of radius~$r$ and support~$v$.
\end{definition}
If~$r<r'$, then an~$(r',v)$-ramification is also an~$(r,v)$-ramification so, in order to show that~$X$ is ramified, it is sufficient to build a sequence~$(v_r)_{r\in\N}$ of pairwise independent vectors such that~$X$ admits an~$(r,v_r)$-ramification for each~$r$.

\begin{remark}[\ti{Subshift of a ramified subshift}]\label{rmk_obs_sub}
Let~$Y$ be a ramified subshift. If the ramifications of~$Y$ can all be taken in a subshift~$X\subsetneq Y$, then~$X$ is ramified as well. Indeed, if~$x\in X$ and a pattern cannot be~$r$-grafted into~$x$ w.r.t.~$Y$, then \emph{a fortiori} it cannot be~$r$-grafted into~$x$ w.r.t.~$X$.

For instance, if~$T$ is a Wang tileset and~$X_T$ has ramifications that can all be built using a proper subset of tiles~$R\subsetneq T$, then for each tileset~$S$ such that~$R\subseteq S\subseteq T$,~$X_S$ is ramified as well.
\end{remark}
%The~$F$-pattern in~$y$ at position~$\lambda v+nu$ cannot be~$r$-grafted in~$y$ at position~$\lambda v$.

Observe that if a subshift is ramified, then it is not strongly irreducible. Actually, showing the lack of strong irreducibility will always be the first step when building ramifications.

%ssssss
\subsection{Examples}
We illustrate the notion of ramification on a simple example.
%sss
\subsubsection{The triangles}\label{sec_triangles_obstructed}
Consider the 2-dimensional SFT~$X_\Delta\subseteq\{\whitetile,\blacktile\}^{\Z^2}$ induced by the following forbidden pattern:
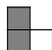
\begin{center}
\includegraphics{Tiles/tiles-4}
\end{center}

A typical configuration is made of black triangles (as in Figure \ref{fig_triangles_not_SI}) distributed on a white background.% shown in Figure \ref{fig_triangles}.

%\begin{figure}[!ht]
%\centering
%\includegraphics{Tiles/tiles-3}
%\caption{The forbidden pattern}\label{fig_triangles}
%\end{figure}
We are going to show that~$X_\Delta$ is ramified. The first step is to prove that it is not strongly irreducible. 
\begin{proposition}
$X_\Delta$ is not strongly irreducible.
\end{proposition}
\begin{proof}
For each~$r\geq 1$, the cell~$(0,0)$ and the column~$\{(-r,k):0\leq k\leq r\}$ are not independent and their distance is~$r$. Indeed, coloring the column in black forces the whole triangle~$T_r=\{(x,y):-r\leq x\leq 0,0\leq y\leq r-x\}$, and in particular the cell~$(0,0)$, to be black as well (see Figure \ref{fig_triangles_not_SI}).

\begin{figure}[!ht]
\centering
\includegraphics{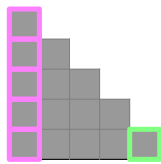}
\caption[Triangles are not strongly irreducible]{The triangle~$T_4$, illustrating that~$X$ is not strongly irreducible: the left column and the right cell are not independent.\qedhere}\label{fig_triangles_not_SI}
\end{figure}
\end{proof}

The black triangles can then be used as building blocks to define ramifications.
\begin{proposition}\label{prop_triangles_obst}
$X_\Delta$ is ramified.
\end{proposition}
\begin{proof}
We build ramifications illustrated in Figure \ref{fig_triangles_obs}. Let~$r\in\N$ and~$V_r=\{(n,1):n\geq r+2\}$, which is an infinite set of pairwise independent vectors. Let~$F=\{(0,0)\}$ and~$u=(0,-1)$. For each~$r$ and each~$v\in V_r$, we build an~$(r,v)$-ramification~$x$ by putting a black triangle~$T_r$ at each multiple of~$v$, everything else being white. Precisely, for each~$\lambda\in\Z$, the cells in~$\lambda v+T_r$ are black, the other cells are white.

\begin{figure}[!ht]
\centering
\includegraphics{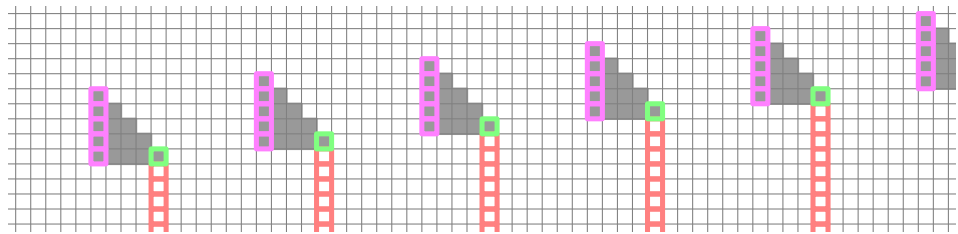}
\caption[Triangles are ramified]{A ramification: the cells~$\lambda v$ are in \green, the cells~$\lambda v+\mu u$,~$\mu>0$, are in \red. In a black triangle, the \purple\ and \green\ regions are correlated, and the content of a \red\ cell cannot be grafted in a \green\ cell without changing the \purple\ region.}\label{fig_triangles_obs}
\end{figure}

First,~$x$ is a valid configuration because the copies of~$T_r$ are separated by white cells, so they do not create forbidden patterns: if~$\lambda\neq \lambda'$ and~$s,t\in T_r$, then~$d(\lambda v+s,\lambda' v+t)=|(\lambda-\lambda')v+(s-t)|\geq |v|-\diam(T_r)\geq (r+2)-r=2$, so two cells from different triangles cannot belong to the forbidden pattern, which has diameter~$1$.

For~$\mu>0$,~$\mu u=(0,-\mu)\notin T_r$ so the cells~$\lambda v+\mu u$ are white. The white cell, seen as an~$F$-pattern, appears at position~$\lambda v+\mu u$ but cannot be~$r$-grafted in~$x$ at any position~$\lambda v$. We have shown that~$x$ is an~$(r,v)$-ramification, and that~$X$ is ramified.
\end{proof}

The argument to prove that a subshift~$X$ is ramified will always follow the same structure as in this example:
\begin{itemize}
\item First show that~$X$ is not strongly irreducible, by identifying for each~$r\in\N$ two regions~$F$ and~$G$ at distance larger than~$r$ from each other that are not independent. We do so by finding valid~$F$-patterns~$\pi_F,\pi'_F$ and a valid~$G$-pattern~$\pi_G$ such that~$\pi_F\cup\pi_G$ is valid but~$\pi'_F\cup\pi_G$ is not valid (in the pictures, $G$ will be colored in \purple\ and~$F$ will be colored in \green),
\item For a suitable choice of~$u$ and~$v$, build an~$(r,v)$-ramification of period~$v$ obtained by placing the pattern~$\pi_F\cup \pi_G$ at each position~$\lambda v$, the pattern~$\pi'_F$ at each position~$\lambda v+\mu u$ with~$\mu>0$ and filling all the other cells to obtain a valid configuration (the translated version of~$F$ at positions~$\lambda v+\mu u$, $\mu>0$, will be colored in \red).
\end{itemize}
%\begin{itemize}
%\item First show that~$X$ is not strongly irreducible, by identifying for each~$r\in\N$ a valid pattern~$\pi_r$ in which two subregions~$F$ and~$G$ at distance~$\geq r$ are not independent, and the~$G$-pattern forbids some valid~$F$-patterns (in the pictures, $G$ will be colored in \purple\ and~$F$ will be colored in \green),
%\item For a suitable choice of~$u$ and~$v$, build a~$(r,v)$-ramification of period~$v$ obtained by repeating the pattern~$\pi_r$ at each multiple of~$v$, filling all the other cells and making sure that the patterns appearing in the translates of~$F$ at positions~$\lambda v+\mu u$, with~$\mu>0$, are all incompatible with the~$G$-pattern appearing in~$\pi$ (the translated of~$F$ at positions~$\lambda v+\mu u$, $\mu>0$, will be colored in \red).
%\end{itemize}
In the pictures, a \red\ region cannot be grafted in a \green\ region without changing the \purple\ region.% In other words, the \purple\ and \green\ region are tied together, the \purple\ pattern forbids the \red\ pattern in the \green\ region.

%sss
\subsubsection{A non-example: the checkerboard}
In order to better understand the notion of ramification and ramified subshift, it is useful to examine a simple subshift which is \emph{not} ramified.

Consider the checkerboard subshift~$X\subseteq\{0,1\}^{\Z^2}$ containing only 2 configurations~$x,y$ defined by~$x_{(i,j)}=i+j\mod 2$ and~$y_{(i,j)}=i+j+1\mod 2$.

First,~$X$ is not strongly irreducible because for each~$r$, the cells~$(0,0)$ and~$(2r+1,0)$ are correlated: each cell can be filled with a~$0$, but not in a common configuration.

However,~$X$ is not ramified. Assume otherwise and let~$x$ be an~$(r,v)$-ramification for some~$r,v$, witnessed by~$F,u$. Choose some arbitrary~$\lambda$ and~$\mu>0$, where~$\mu$ is \emph{even} (for instance,~$\lambda=0,\mu=2$). The~$F$-pattern appearing in~$x$ at position~$\lambda v+\mu u$ coincides with the~$F$-pattern appearing in~$x$ at position~$\lambda v$, so it can be~$r$-grafted (and even~$0$-grafted) at that position, contradicting the definition of a ramification.

%ssssss
\subsection{An obstruction to being in \texorpdfstring{$\gen$}{L0}}
We state the main result of this section, which gives a way of showing that a subshift is not in~$\gen$, when that subshift is weakly mixing.
\begin{theorem}[\ti{An obstruction to being in $\gen$}]\label{thm_obs_C0}
Let~$X$ be a~$\Z^2$-subshift. If~$X$ is ramified and weakly mixing, then~$X\notin \gen$.
\end{theorem}

The rest of this section is devoted to the proof of this result.
%\begin{remark}\label{rmk_wmix}
%The weak mixing assumption cannot be dropped from the statement.
%
%When a subshift~$X$ is obstructed, the ramifications that are built to prove this fact induce a subshift~$Y\subseteq X$, defined as the smallest subshift containing these ramifications. $Y$ is obstructed as well by Remark \ref{rmk_obs_sub}. In many cases,~$Y$ is countable because the countable set of ramifications and their shifted versions only generate a countable set of limit configurations (it is the case in the proof of Proposition \ref{prop_triangles_obst}). In such a case,~$Y\in\gen$, which is possible because~$Y$ is not weakly mixing.
%\end{remark}

%ssssss
\subsubsection{Uniform case}
We first show that a ramified subshift cannot be a factor of a product of fullshifts and periodic shifts. To do so, we prove that finite products of fullshifts and periodic shifts are not ramified, and that factor maps send non-ramified subshifts to non-ramified subshifts.

The ``uniformity'' in the title of this section refers to the fact that factor maps are uniform local functions, with the same rule in each cell. In the next section, we will consider the case of arbitrary, non-uniform local functions. Strictly speaking, the results from this section are implied by the results in the next section, but we prefer to present both because the first ones are easier to understand.

\begin{theorem}\label{thm_obstr_uniform}
Let~$X$ be a~$\Z^2$-subshift. If~$X$ is ramified, then~$X$ not a factor of~$A^{\Z^2}\times \prod_{i=1}^kA^{\Z^2/u_i\Z}$, with each~$u_i\neq 0$. 
\end{theorem}

We are going to need the following lemma. In the plane, given lines~$L_0,\ldots,L_k$ going through the origin and a direction~$v$ which is not colinear to any~$L_i$, it is obvious that there exists a line supported by~$v$ intersecting each~$L_i$, which does not contain the origin. The next lemma states that it also holds when all the coefficients are integers.

\begin{lemma}\label{lem_line2}
Let~$u_0,u_1,\ldots,u_k\in\Z^2$, with~$k\geq 1$. Let~$v\in\Z^2$ be independent of each~$u_i$ individually. There exist~$\lambda_1, \lambda_2,\ldots,\lambda_k$ and~$\mu>0$ such that each~$\mu u_0-\lambda_iv$ is a multiple of~$u_i$.
%Moreover, given a number~$d>0$, we can make sure that~$|p-\lambda_iv|\geq d$ for all~$i$.
\end{lemma}
\begin{proof}
We show that there exist non-zero integers~$n_0,\ldots,n_k$ such that the pairwise differences of~$n_iu_i$ are multiples of~$v$. If~$n_0<0$ then we replace each~$n_i$ by~$-n_i$. We then take~$\mu=n_0$ and~$\lambda_i$ such that~$n_0u_0-n_iu_i=\lambda_i v$.

First consider the case when~$k=1$. As~$\Z^2$ has rank~$2$ as a~$\Z$-module, there is a non-trivial combination~$n_0u_0-n_1u_1+nv=0$. As~$v$ is independent of each~$u_i$ individually, one has~$n_0\neq 0$ and~$n_1\neq 0$, and the difference of~$n_0u_0$ and~$n_1u_1$ is a multiple of~$v$.

Let us now consider the general case~$k\geq 2$. For each~$i$ from~$1$ to~$k$, the previous argument applied to~$u_0$ and~$u_i$ provides non-zero coefficients~$n^i_0$ and~$n^i_1$ such that~$n^i_0u_0-n^i_1u_i$ is a multiple of~$v$. Let~$n_0=\prod_j n^j_0$ and~$n_i=n^i_1\prod_{j\neq i}n^j_0$. All these coefficients are non-zero, and~$n_0u_0-n_iu_i=(n^i_0u_0-n^i_1u_i)\prod_{j\neq i}n^j_0$ is a multiple of~$v$.
%Finally, if~$d>0$ is given, then multiply~$p$ and each~$\lambda_i$ by~$d$, so that~$|dp-d\lambda_iv|=d[p-\lambda_iv|\geq d$.
\end{proof}

\begin{proposition}\label{prop_product_obstructed}
Let~$X=A^{\Z^2}\times \prod_{i=1}^kA^{\Z^d/u_i\Z}$ be a finite product of fullshifts and periodic shifts. $X$ is not ramified.
\end{proposition}
\begin{proof}
Let~$X=\prod_{i=0}^kX_i$ where~$X_0$ is a fullshift and for~$i\geq 1$,~$X_i$ is a periodic shift of period~$u_i\neq 0$. A configuration~$x\in X$ is~$(x_0,\ldots,x_k)$ with each~$x_i\in X_i$. Let~$A$ be the finite alphabet such that~$X\subseteq A^{\Z^2}$.

Let~$v$ be independent of each~$u_i$ and~$x$ be a~$(0,v)$-ramification, with~$u,F$. We apply Lemma \ref{lem_line2}, giving~$\lambda_1,\ldots,\lambda_k$ and~$\mu>0$ such that~$\mu u-\lambda_iv$ is a multiple of~$u_i$. Let~$\lambda_0=0$,~$L^+=\max_{0\leq i\leq k}\lambda_i$, $L^-=\min_{0\leq i\leq k}\lambda_i$ and~$L=L^+-L^-+1$.

We define a coloring of~$[0,N]$ where~$N=N(A^F,L)$ is given by Van der Waerden's theorem (Theorem \ref{thm_van_der_waerden}). Assign to~$n$ the~$F$-pattern~$\pi_n\in A^F$ appearing at position~$nv$ in~$x$. By Van der Waerden's theorem, this coloring admits a monochromatic arithmetic progression of length~$L$, i.e.~there exist~$a$ and~$b>0$ such that all~$(a+\lambda b)$ have the same color,~$\lambda \in [L^-,L^+]$. Let~$\pi$ be the common~$F$-pattern appearing at these positions~$(a+\lambda b)v$. In particular, if~$p_i=(a+\lambda_i b)v$ and~$q\in F$ then
\begin{equation}\label{eq_same_pattern}
x(p_i+q)=\pi(q)
\end{equation}
does not depend on~$i$.

Let~$p=av+b\mu u$ and~$\xi$ be the~$F$-pattern appearing at position~$p$ in~$x$. We claim that~$\xi$ can be~$0$-grafted into~$x$ at position~$p_0$. In other words, if we define~$x'\in A^E$ by
\begin{equation*}
x'(p_0+q)=\begin{cases}
\xi(q)=x(p+q)&\text{if }q\in F,\\
x(p_0+q)&\text{otherwise,}
\end{cases}
\end{equation*}
then we show that~$x'$ belongs to~$X$. Note that~$x'_0\in X_0$ because~$X_0$ is the fullshift. We claim that for~$i\geq 1$, one has~$x'_i=x_i\in X_i$, from which it will follow that~$x'\in X$.

We use the fact that~$p-p_i=b(\mu u-\lambda _iv)$ which is a multiple of~$u_i$ and that~$x_i\in X_i$ is~$u_i$-periodic, so~$x_i(p+q)=x_i(p_i+q)$ for all~$q$.

If~$q\notin F$, then~$x'(p_0+q)=x(p_0+q)$ by definition of~$x'$. If~$q\in F$ and~$i\geq 1$, then
\begin{align*}
x'_i(p_0+q)&=x_i(p+q)&\text{by definition of $x'$},\\
&=x_i(p_i+q)&\text{as $p-p_i$ is a multiple of~$u_i$},\\
&=x_i(p_0+q)&\text{as $x(p_i+q)=\pi(q)=x(p_0+q)$ by \eqref{eq_same_pattern}.}
\end{align*}
As a result,~$x'_i=x_i\in X_i$ for~$i\geq 1$, so~$x\in X$.

We have shown that the~$F$-pattern appearing at position~$p=av+b\mu u$ in~$x$ can be~$0$-grafted into~$x$ at position~$p_0=av$, contradicting the assumption that~$x$ is a~$(0,v)$-ramification.
\end{proof}

Being ramified is indeed a conjugacy invariant, as shown by the next result.
\begin{proposition}\label{prop_factor_obstructed}
If~$Y$ is a factor of~$X$ and~$Y$ is ramified, then~$X$ is ramified as well.
\end{proposition}
\begin{proof}
Let~$f:X\to Y$ be a factor map and~$k\in\N$ be such that each~$p\in\Z^2$ is determined by its neighborhood~$\Nei(p,k)$. For~$r\in\N$, let~$v$ be such that~$Y$ admits an~$(r+2k,v)$-ramification~$y\in Y$, witnessed by~$u$ and~$F$. Let~$x\in X$ be such that~$f(x)=y$. We claim that~$x$ is an~$(r,v)$-ramification, witnessed by~$u$ and~$W:=\Nei(F,k)$.

Let~$\lambda$ and~$\mu>0$ be integers. Assume for a contradiction that the~$W$-pattern appearing at position~$\lambda v+\mu u$ in~$x$ can be grafted into~$x$ at position~$\lambda v$, let~$x'\in X$ be the result of this graft and~$y'=f(x')$. As~$x'$ coincides with~$x$ outside~$\Nei(\lambda v+W,r)=\Nei(\lambda v+F,r+k)$,~$y'$ coincides with~$y$ outside~$\Nei(\lambda v+F,r+2k)$. The~$F$-pattern appearing at position~$\lambda v$ in~$y'$ is the~$F$-pattern appearing at position~$\lambda v+\mu u$ in~$y$, so~$y'$ witnesses that this pattern can be~$(r+2k)$-grafted into~$y$, contradicting the assumption.

Given~$r\in\N$, there exist infinitely many vectors~$v$ such that~$Y$ admits an~$(r+2k,v)$-ramification, and~$X$ admits an~$(r,v)$-ramification for each such~$v$.
\end{proof}

These results imply Theorem \ref{thm_obstr_uniform}.

%ssss
\subsubsection{Non-uniform case}
Our next goal is to extend Theorem \ref{thm_obstr_uniform} to show that a ramified subshift cannot be the image of a product of fullshifts and periodic shifts by a local function. Local functions do not necessarily play well with ramifications, so we need to refine the invariant property of being ramified. The idea is that local functions are actually uniform on arbitrarily large regular grids, by Gallai-Witt's theorem (Theorem \ref{thm_gallai_witt}). Using this observation, we show that they preserve a finite version of ramifications. We start defining this finite version, which is unfortunately rather technical and non-intuitive, but makes the argument go through.

\begin{definition}
Let~$r\in\N$ and~$v\in\Z^2$. A subshift~$X$ is \textbf{finitely~$(r,v)$-ramified} if there exist~$u\in\Z^2$ and~$F\subseteq\Z^2$ such that for all~$K\in\N$, there exist~$x\in X$ and~$\beta>0$ such that for all~$\lambda\in[-K,K]$ and~$\mu\in [1,K]$, the~$F$-pattern appearing at position~$\lambda \beta v+\mu \beta u$ in~$x$ cannot be~$r$-grafted into~$x$ at position~$\lambda \beta v$.

A subshift~$X$ is \textbf{finitely ramified} if for all~$r\in\N$ and for infinitely many pairwise independent vectors~$v$,~$X$ is finitely~$(r,v)$-ramified.
\end{definition}

\begin{remark}[\ti{Ramified implies finitely ramified}]
If~$X$ admits an~$(r,v)$-ramification~$(x_0,u,F)$, then~$X$ is finitely~$(r,v)$-ramified. Indeed, given~$K\in\N$, take~$x=x_0$ and~$\beta=1$. Therefore, if~$X$ is ramified, then~$X$ is finitely ramified.
\end{remark}

\begin{theorem}\label{thm_obstr_non_uniform}
Let~$X$ be a~$\Z^2$-subshift. If~$X$ is finitely ramified, then~$X$ is not the image of any~$A^{\Z^2}\times \prod_{i=1}^k\Z^2/u_i\Z$ by a local function. 
\end{theorem}

This result will directly follow from the next results, which refine Proposition \ref{prop_product_obstructed} and \ref{prop_factor_obstructed}.

\begin{proposition}
Let~$X$ be a finite product of fullshifts and periodic shifts. $X$ is not finitely ramified.
\end{proposition}
\begin{proof}
If~$X$ is a fullshift, then it is strongly irreducible so it is not finitely ramified. We now assume that~$X=\prod_{i=0}^kX_i$ where~$k\geq 1$,~$X_0$ is a fullshift and for~$1\leq i\leq k$,~$X_i$ is a periodic shift of period~$u_i\neq 0$. A configuration~$x\in X$ is~$(x_0,\ldots,x_k)$ with each~$x_i\in X_i$. Let~$A$ be the finite alphabet such that~$X\subseteq A^{\Z^2}$.

Let~$v$ be independent of each~$u_i$ and such that~$X$ is finitely~$(0,v)$-ramified, using~$u,F$. Let~$u_0=u$. We apply Lemma \ref{lem_line2}, giving~$\lambda_1,\ldots,\lambda_k$ and~$\mu>0$ such that each~$\mu u-\lambda_iv$ is a multiple of~$u_i$. Let~$\lambda_0=0$,~$L^+=\max_i\lambda_i$, $L^-=\min_i\lambda_i$ and~$L=L^+-L^-+1$.

Let~$N=N(A^F,L)$ be given by Van der Waerden's theorem (Theorem \ref{thm_van_der_waerden}) and~$K=N\mu$. There is~$x\in X$ and~$\beta>0$ witnessing the finite~$(0,v)$-ramifications.

We define a coloring of~$[0,N]$. Assign to each~$n\in [0,N]$ the~$F$-pattern~$\pi_n\in A^F$ appearing at position~$n\beta v$ in~$x$. By the Van der Waerden's theorem, this coloring admits a monochromatic arithmetic progression of length~$L$, i.e.~there exist~$a$ and~$b\geq 1$ such that all~$(a+\lambda b)$ have the same color,~$\lambda \in [L^-,L^+]$ (in the trivial case when~$L=1$, we can take~$b=1$). Let~$\pi$ be the common~$F$-pattern appearing at these positions~$(a+\lambda b)\beta v$. In particular, if~$p_i=(a+\lambda_i b)\beta v$ and~$q\in F$ then~$\pi(q)=x(p_i+q)$ does not depend on~$i$.

Observe that~$a\in [-K,K]$ and~$b\mu\in [1,K]$: indeed,~$a=a+\lambda_0 b\in [0,N]$ and~$b\leq N/(L-1)\leq N$ (if~$L=1$ then~$b=1=N$).

Let~$p=a\beta v+b\mu \beta u$ and~$\xi$ be the~$F$-pattern appearing at position~$p$ in~$x$. We claim that~$\xi$ can be~$0$-grafted into~$x$ at position~$p_0$. In other words, if we define~$x'\in A^E$ by
\begin{equation*}
x'(p_0+q)=\begin{cases}
x(p+q)&\text{if }q\in F,\\
x(p_0+q)&\text{otherwise,}
\end{cases}
\end{equation*}
then~$x'$ belongs to~$X$. Note that~$x'_0\in X_0$ because~$X_0$ is the fullshift. We claim that for~$i\geq 1$, one has~$x'_i=x_i\in X_i$, from which it follows that~$x'\in X$.

We use the fact that~$p-p_i=b(\mu \beta u-\lambda _i \beta v)$ which is a multiple of~$u_i$, so~$x_i(p+q)=x_i(p_i+q)$ for all~$q$.

If~$q\notin F$, then~$x'(p_0+q)=x(p_0+q)$ by definition of~$x'$. If~$q\in F$ and~$i\geq 1$, then
\begin{align*}
x'_i(p_0+q)&=x_i(p+q)&\text{by definition of $x'$},\\
&=x_i(p_i+q)&\text{as $p-p_i$ is a multiple of~$u_i$},\\
&=x_i(p_0+q)&\text{as $x(p_i+q)=\pi(q)=x(p_0+q)$.}
\end{align*}
As a result,~$x'_i=x_i\in X_i$ for~$i\geq 1$, so~$x\in X$.

We have shown that the~$F$-pattern appearing at position~$p=a\beta v+b\mu \beta u$ in~$x$ can be~$0$-grafted into~$x$ at position~$p_0=a\beta v$.

One has~$|a|\leq K$ and~$|b\mu|\leq K$ so it contradicts the property of finite~$(0,v)$-ramifications.
\end{proof}

\begin{proposition}
Let~$X,Y$ be subshifts, such that~$Y$ is the image of~$X$ by a local function. If~$Y$ is ramified, then~$X$ is finitely ramified.
\end{proposition}
\begin{proof}
Let~$f:X\to Y$ of radius~$k$. Let~$r\in\N$ and~$v$ be such that~$Y$ has an~$(r+2k,v)$-ramification~$y$, witnessed by~$u$ and~$F$. We define~$W=\Nei(F,k)$, which is the region that determines~$F$, and show that~$X$ is finitely~$(r,v)$-ramified, witnessed by~$u$ and~$W$.

Let~$K\in\N$ and~$G_K=\{\lambda v+\mu u:\lambda \in [-K,K],\mu\in [1,K]\}$.

To each~$p\in\Z^2$ we associate the rule that determines the~$F$-pattern appearing at position~$p$ in~$f(x)$ from the~$W$-pattern appearing at position~$p$ in~$x$. If~$A$ and~$B$ are the finite alphabets underlying~$X$ and~$Y$ respectively, then a rule is an element of the finite set~$(B^F)^{A^W}$. By Gallai-Witt's theorem (Theorem \ref{thm_gallai_witt}), there exist~$a\in\Z^2$ and~$\alpha>0$ such that all the cells in~$a+\alpha G_K$ have the same rule. The shift-invariance of~$X$ and~$Y$ enables us to get rid of~$a$ as follows: the function~$g=\sigma^a\circ f\circ \sigma^{-a}$ is a surjective local function of radius~$k$, it surjectively sends~$X$ to~$Y$, and all the cells in~$\alpha G_K$ have the same rule.

Let~$x\in X$ be such that~$g(x)=y$, we show that~$x$ and~$\alpha$ satisfy the finite ramification property for~$X$.

Assume for a contradiction that for some~$\lambda \in [-K,K]$ and~$\mu\in [1,K]$, the~$W$-pattern appearing at position~$p=\lambda \alpha v+\mu\alpha u$ can be~$r$-grafted into~$x$ at position~$p_0=\lambda\alpha v$, let~$x'$ be the result of this~$r$-graft and define~$y'=g(x')$. Note that~$p,p_0\in \alpha G_K$ so the function~$g$ has the same rules at positions~$p$ and~$p_0$. This property and the fact that the~$W$-pattern at position~$p$ in~$x$ coincides with the~$W$-pattern at position~$p_0$ in~$x'$ imply that the~$F$-pattern~$\pi$ appearing at position~$p$ in~$y$ coincides with the~$F$-pattern at position~$p_0$ in~$y'$. As~$x'$ coincides with~$x$ outside~$\Nei(p_0+W,r)=\Nei(p_0+F,r+k)$ and~$g$ has radius~$k$,~$y'=g(x')$ coincides with~$y=g(x)$ outside~$\Nei(p_0+F,r+2k)$. As~$y'\in Y$, it shows that the~$\pi$ can be~$(r+2k)$-grafted into~$y$ at position~$p_0$, which contradicts the assumption that~$y$ is an~$(r+2k,v)$-ramification for~$Y$.

Finally, given~$r$ there are infinitely many pairwise independent vectors~$v$ such that~$Y$ has~$(r+2k,v)$-ramifications, and we have shown that~$X$ is finitely~$(r,v)$-ramified for each such~$(r,v)$. As a result,~$X$ is finitely ramified.
\end{proof}

We have all the ingredients to prove Theorem \ref{thm_obs_C0}.
\begin{proof}[Proof of Theorem \ref{thm_obs_C0}]
As~$X$ is ramified,~$X$ is not the image of a local function by Theorem \ref{thm_obstr_non_uniform}. As~$X$ is weakly mixing, it does not belong to~$\gen$ by Proposition \ref{prop_wmix_local}.
\end{proof}

%sss
\subsection{Applications}\label{sec_app_obstructions}
We are going to apply Theorem \ref{thm_obs_C0} to two subshifts.
%sss
\subsubsection{The triangles}\label{sec_triangles}
We start with the subshift~$X_\Delta$, already analyzed in Section \ref{sec_triangles_obstructed}.
\begin{theorem}
$X_\Delta$ is ramified and mixing, therefore it is not in~$\gen$.
\end{theorem}
\begin{proof}
We have already seen that it is ramified (Proposition \ref{prop_triangles_obst}), let us prove that it is mixing. Let~$T_n=\{(x,y)\in\Z^2:x,y\geq 0,x+y\leq n\}$. Any valid~$T_n$-pattern can be extended to a valid configuration by making all the other cells white: one easily checks that it does not create a forbidden pattern. If~$\pi,\pi'$ are valid~$T_n$-patterns then for all~$p\in\Z^2$ such that~$|p|\geq \diam(T_n)+2=n+2$,~$\pi\cup\sigma^p(\pi')$ is valid because~$d(T_n,T_n-p)\geq 2$ so the extension of~$\pi\cup\sigma^p(\pi')$ with white cells is a valid configuration: the forbidden pattern has diameter~$1$, so it cannot intersect both~$T_n$ and~$T_n-p$.
\end{proof}

%sss
\subsubsection{Domino tilings}\label{sec_dominos}
We consider the \textbf{domino tileset} shown in Figure \ref{fig_domino}. It induces a very classical tiling which is a particular case of the dimer model in statistical mechanics. The problem of generating random configurations has been studied in \cite{JPS98,Propp03,JRV06} among others.
\begin{figure}[!ht]
\centering
\includegraphics{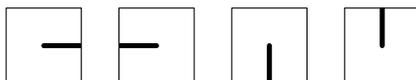}
\caption{The domino tileset}\label{fig_domino}
\end{figure}

\begin{theorem}
The domino subshift is ramified and mixing, so it does not belong to~$\gen$.
\end{theorem}
\begin{proof}
Let~$X$ be the domino subshift. First, $X$ is not strongly irreducible, as illustrated in Figure \ref{fig_domino_not_SI}. For~$k\in\N$, the horizontal line~$H_k=[-k,k]\times\{0\}$ filled with an alternation of \smalltile{0100} and \smalltile{0001} determines the content of the upper triangle~$\{(i,j):j\geq 1, |i|+j\leq k+1\}$, in particular its tip~$T_k=(0,k+1)$. Therefore,~$T_k$ and~$H_k$ are correlated and~$d(T_k,H_k)=k+1$ is arbitrarily large.
\begin{figure}[!ht]
\centering
\includegraphics{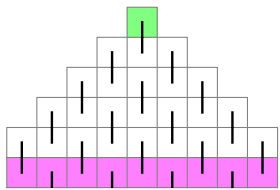}
\caption[Dominos are not strongly irreducible]{The domino subshift is not strongly irreducible: the \purple\ pattern determines the content of the \green\ cell.}\label{fig_domino_not_SI}
\end{figure}

Figure \ref{fig_domino_obstruction} shows a ramification. One can build an~$(r,v)$-ramification for any~$r$ and~$v=(r,r+2)$ by placing diamonds~$\{(i,j):|i|+|j+0.5|\leq r+0.5\}$ of width~$2r+1$ made of vertical tiles \smalltile{0100} and \smalltile{0001} with their tips at the multiples of~$v$, and filling each remaining horizontal half-line with the horizontal tiles \smalltile{0010} and \smalltile{1000}. The distance between the tip and the lower middle row of the diamond is~$r+1$. We take~$u=(-1,1)$ and~$F=\{(0,0)\}$. All the cells~$\lambda v+\mu u$, with~$\mu>0$, are filled with horizontal tiles, which cannot be~$r$-grafted at~$\lambda v$.
\begin{figure}[!ht]
\centering
\includegraphics{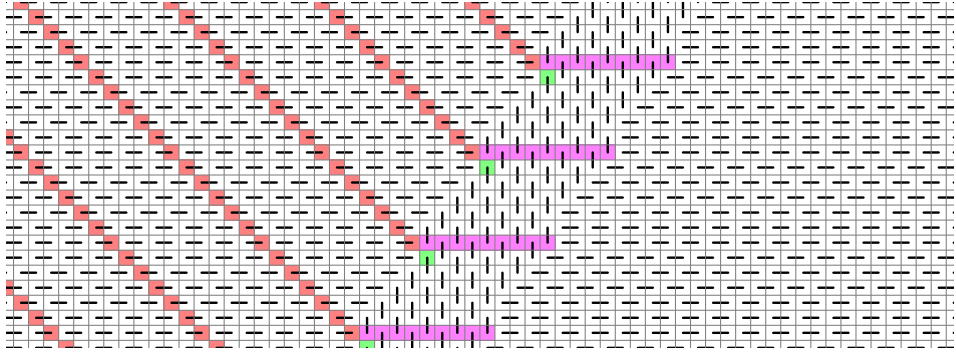}
\caption[Dominos are ramified]{An~$(r,v)$-ramification with~$r=4$ and~$v=(4,6)$,~$u=(-1,1)$ and~$F=\{(0,0)\}$}\label{fig_domino_obstruction}
\end{figure}

We finally prove that~$X$ is mixing. Let~$D_k$ be the diamond~$\{(i,j)\in\Z^2:|i|+|j|\leq k\}$. We show that if~$|p|\geq 2k+3$, then~$D_k$ and~$p+D_k$ are independent.

Let~$\pi$ be a valid~$D_k$-pattern. We first extend~$\pi$ to a~$D_{k+1}$-pattern as follows. Let~$(i,j)\in D_{k+1}\setminus D_k$. If~$i\geq 1$, then we do not use \smalltile{0010}; if~$i\leq -1$, then we do not use \smalltile{1000}; if~$i=0$ then we do not use both tiles. $\pi'$ has the property that any vertical edge on its boundary is white. We can therefore extend~$\pi'$ to a valid pattern~$\xi$ on~$[-k-1,k+1]\times \Z$ by using \smalltile{0100} and \smalltile{0001} only, in a deterministic way. The boundaries of~$\xi$ are white.

Let now~$p=(i,j)$ be such that~$|p|\geq 2k+3$, and let~$\pi,\pi'$ be valid patterns on~$D_k$ and~$p+D_k$ respectively. There are two cases:~$|i|\geq 2k+3$ or~$|j|\geq 2k+3$.

First assume that~$|i|\geq 2k+3$. One can extend~$\pi,\pi'$ to valid patterns~$\xi,\xi'$ on~$[-k-1,k+1]\times \Z$ and~$[i-k-1,i+k+1]\times \Z$ respectively, with white boundaries. They do not intersect as~$|i|\geq 2k+3$. The rest of the plane can be filled with \smalltile{0100} and \smalltile{0001}.

Now assume that~$|j|\geq 2k+3$. It reduces to the first case by applying a rotation by 90 degrees.
\end{proof}

%SSSSSSS
\section{Transitions}\label{sec_transitions}
We now develop a technique to prove that a subshift does not belong to~$\nar$. There are many subshifts in which typical configurations exhibit a transition between two regions that use different sets of tiles. We show that this type of subshift does not belong to~$\nar$.

%ssss
\subsection{An obstruction to being in \texorpdfstring{$\nar$}{L1}}
The simplest example of a subshift witnessing a transition is the one-dimensional SFT on the alphabet~$\{\raisebox{-.5mm}{\includegraphics{Tiles/tiles-7}},\raisebox{-.5mm}{\includegraphics{Tiles/tiles-8}}\}$ with forbidden pattern \raisebox{-.5mm}{\includegraphics{Tiles/tiles-6}}.

We denote by~$X_{\Zb}$ the induced~$\Z$-subshift, because it is conjugate to~$\Zb=\Z\cup \{-\infty,+\infty\}$ whose topology is generated by the basis of open sets~$\{n\}$,~$[n,+\infty]$ and~$[-\infty,n]$ for~$n\in\Z$, and with the continuous~$\Z$-action
\begin{align*}
p\cdot n&=p+n\text{ for finite $n$},\\
p\cdot (+\infty)&=+\infty,\\
p\cdot (-\infty)&=-\infty.
\end{align*}

The conjugacy map~$F:\Zb\to X_{\Zb}$ sends~$n\in\Zb$ to the configuration which is white on the cells~$p<-n$ and black on the cells~$p\geq -n$. In particular, it sends~$-\infty$ to the white configuration and~$+\infty$ to the black configuration.
%
%The simplest example of such a transition is obtained with the following 1D tileset:
%\begin{figure}[!ht]
%\centering
%\includegraphics{tilesets3-1}
%\caption{The tileset inducing~$\Zb$}
%\end{figure}
%
%We denote by~$X_{\Zb}$ the induced subshift, because it is conjugate to~$\Zb=\Z\cup \{-\infty,+\infty\}$ whose topology is generated by the basis of open sets~$\{n\}$,~$[n,+\infty]$ and~$[-\infty,n]$ for~$n\in\Z$, and with the continuous~$\Z$-action
%\begin{align*}
%p\cdot n&=p+n\text{ for finite $n$},\\
%p\cdot (+\infty)&=+\infty,\\
%p\cdot (-\infty)&=-\infty.
%\end{align*}
%
%The conjugacy map~$F:\Zb\to X_{\Zb}$ sends~$-\infty$ to the configuration with only \smalltile{1010},~$+\infty$ to the configuration with only \smalltile{0000} and~$n\in\Z$ to the configuration with the transition tile \smalltile{0010} at position~$-n$.

The next result shows in particular that~$X_{\Zb}$ is not the image of a narrow function. It relies of the following simple observation: given a continuous function, if two output regions do not communicate, i.e.~are determined by disjoint input regions, then they are independent. We recall that, once a continuous function~$f:A^\N\to B^E$ is fixed, to each output position~$p\in E$ is associated its input window~$\W_f(p)\subseteq\N$. More generally, to each output region~$F\subseteq E$ we associate its input window~$\W_f(F)=\bigcup_{p\in F}\W_f(p)$. The values of an input~$x\in A^\N$ on~$\W_f(F)$ determined the value of the output~$f(x)$ on~$F$.
\begin{lemma}
Let~$f:A^\N\to B^E$ be a narrow function and~$X$ be the image of~$f$. Let~$F,G\subseteq E$. If~$\W_f(F)\cap \W_f(G)=\emptyset$, then~$F$ and~$G$ are independent w.r.t.~$X$.
\end{lemma}
\begin{proof}
Let~$\pi_F,\pi_G$ be a valid patterns on~$F,G$ respectively. There exist a~$\W_f(F)$-pattern~$\xi_F$ and a~$\W_f(G)$-pattern~$\xi_G$ such that~$f([\xi_F])\subseteq [\pi_F]$ and~$f([\xi_G])\subseteq [\pi_G]$. As~$\W_f(F)$ and~$\W_f(G)$ are disjoint, there exists~$x\in A^\N$ extending~$\xi_F\cup \xi_G$, so~$f(x)$ extends~$\pi_F\cup \pi_G$ which is then valid. Therefore~$F$ and~$G$ are independent.
\end{proof}

\begin{proposition}\label{prop_zbar_narrow}
Let~$A$ be a finite alphabet. If~$f:A^\N\to X_{\Zb}$ is a narrow function, then its image is finite. Moreover, if~$f$ is~$r$-narrow, then
\begin{equation*}
|\im{f}|\leq |A|^\frac{r(r+1)}{2}.
\end{equation*}
In particular,~$X_{\Zb}$ is not the image of a narrow function.
\end{proposition}

\begin{proof}
We prove the result by induction on~$r$. Let~$b(r)=|A|^{\frac{r(r+1)}{2}}$. For~$r=0$, a~$0$-narrow function is contant, so~$\im{f}=1=b(0)$.

Let~$r\geq 1$ and assume by induction that the result holds for~$r-1$. Let~$f:A^\N\to X_{\Zb}$ and let~$E=\{p\in\Z:\exists x,y\in \im{f},x(p)\neq y(p)\}$ be the set of cells to which~$f$ assigns several values. If~$E=\emptyset$, then~$f$ is constant so~$|\im{f}|=1\leq b(r)$. Assume from now one that~$E\neq\emptyset$.

Let us show that for every pair of distinct elements~$p,q\in E$, the cells~$p$ and~$q$ are correlated in~$\im{f}$. We can assume w.l.o.g.~that~$p<q$. As~$p,q\in E$, there exist~$x,y\in \im{f}$ such that~$x(p)=\raisebox{-.5mm}{\includegraphics{Tiles/tiles-8}}$ and~$y(q)=\raisebox{-.5mm}{\includegraphics{Tiles/tiles-7}}$. However, there is no~$z\in \im{f}$ such that~$z(p)=x(p)$ and~$z(q)=y(q)$. We have shown that~$p$ and~$q$ are correlated in~$\im{f}$. Therefore, the input windows~$\W_f(p)$ and~$\W_f(q)$ of~$p$ and~$q$ w.r.t.~$f$ must have non-empty intersection. Note that for~$p\notin E$,~$\W_f(p)=\emptyset$ because~$x(p)$ is constant for~$x\in \im{f}$.

Fix some~$p\in E$ and let~$\pi$ be an input~$\W_f(p)$-pattern. We define~$f_\pi:A^{\N}\to\Zb$ by~$f_\pi(x)=f(x')$ where~$x'$ is obtained from~$x$ by replacing the content of~$\W_f(p)$ by~$\pi$. The function~$f_\pi$ is~$(r-1)$-narrow. Indeed, every output cell~$q$ is determined by~$\W_f(q)\setminus \W_f(p)$, whose size is at most~$r-1$ as it is either empty, or strictly contained in~$\W_f(q)$ which has size at most~$r$.

Therefore, by induction hypothesis,~$\im{f_\pi}\leq b(r-1)$. One has~$\im{f}=\bigcup_\pi\im{f_\pi}$, and there are~$|A|^{|\W_f(p)|}\leq |A|^r$ many~$\W_f(p)$-patterns~$\pi$, so~$|\im{f}|\leq |A|^rb(r-1)=b(r)$.
\end{proof}

We now state the main technique that will be used to show that a subshift is not in~$\nar$, and which is a direct consequence of the previous results.

\begin{theorem}[\ti{An obstruction to being in~$\nar$}]\label{thm_notin_nar}
Let~$X$ be a~$\Z^d$-subshift. If~$X$ is weakly mixing and~$\Zb$ is a weak factor of~$X$, then~$X\notin\nar$.
\end{theorem}
\begin{proof}
As~$X$ is weakly mixing we can apply Proposition \ref{prop_wmix_narrow}, so we only need to show that~$X$ is not the image of a narrow function. By assumption, there exists weak factor map~$F:X\to\Zb$, which is in particular a narrow function. If~$X$ was the image of a narrow function~$G:A^\N\to X$, then~$\Zb$ would be the image of the narrow function~$F\circ G$, contradicting Proposition \ref{prop_zbar_narrow}.
\end{proof}

%ssssss
\subsection{Application}\label{sec_app_transitions}
The simplest subshift to which Theorem \ref{thm_notin_nar} can be applied is the two-dimensional version of the SFT discussed in the previous section (note that the one-dimensional version is countable, so it belongs to~$\gen\subseteq\nar$).

%We consider the two-dimensional SFT on the alphabet~$\{\raisebox{-1mm}{\includegraphics{Tiles/tiles-7}},\raisebox{-1mm}{\includegraphics{Tiles/tiles-8}}\}$ with forbidden pattern \raisebox{-1mm}{\includegraphics{Tiles/tiles-9}}. 

\begin{proposition}\label{prop_ZbarZ}
Let~$X$ be the~$\Z^2$-SFT on the alphabet~$\{\whitetile,\blacktile\}$ with forbidden pattern \raisebox{-.5mm}{\includegraphics{Tiles/tiles-6}}. $X$ weak factors to~$\Zb$ and is weakly mixing, therefore~$X\notin\nar$.
\end{proposition}
\begin{proof}
The function that extracts the first row of a configuration of~$X$ is a weak factor map to the~$\Z$-subshift~$X_{\Zb}$, via the homomorphism~$\varphi:\Z\to\Z^2$ defined by~$\varphi(p)=(p,0)$.

It is weakly mixing, because two regions that do not intersect a common row are independent. In particular for each~$n$,~$S_n$ and~$(0,n)+S_n$ are independent.

We then apply Theorem \ref{thm_notin_nar}, implying~$X\notin\nar$.
\end{proof}
%
%
%We now consider the SFT
%\begin{proposition}
%Let~$X$ be the~$\Z^2$-SFT on the alphabet~$\{\whitetile,\blacktile\}$ with forbidden patterns
%\begin{center}
%\raisebox{1.5mm}{\includegraphics{Tiles/tiles-6}}$\quad$ \includegraphics{Tiles/tiles-12}$\quad$ \includegraphics{Tiles/tiles-13}
%\end{center}
%$X$ weak factors to~$\Zb$ and is weakly mixing, therefore~$X\notin\nar$.
%\end{proposition}

%ssssssssss
\subsection{Other subshifts}\label{sec_argument_nar}
In Section \ref{sec_obstructions} we identified two subshifts that do not belong to~$\gen$. We expect them not to belong to~$\nar$ either, but have not been able to prove it so far. The technique used in this section, namely Theorem \ref{thm_notin_nar}, cannot be applied to them because they are all mixing and Proposition \ref{prop_weak_factor_mixing} prevents~$\Zb$ from being a weak factor of a mixing subshift. Indeed,~$\Zb$ is neither periodic nor mixing.

We present two other arguments showing that a subshift does not belong to~$\nar$.

%sss
\subsubsection{Stacks}
Consider the SFT~$X\subseteq \{0,1,2\}^{\Z^2}$ with forbidden patterns~$\pattern{1\\0}$ and~$\pattern{2&0}$ and~$\pattern{2&1}$. A valid configuration is obtained by filling certain rows with~$2$'s and by filling each portion of each column lying between two occurrences of~$2$, by a stack of~$1$'s with a stack of~$0$'s above it.

Let us first show that~$\Zb$ is not a weak factor of~$X$, implying that Theorem \ref{thm_notin_nar} cannot be applied. 
\begin{proposition}
Let~$X,Y$ be subshifts, and assume that~$Y$ is a weak factor of~$X$. If the strongly periodic configurations of~$X$ are dense in~$X$, then the strongly periodic configurations of~$Y$ are dense in~$Y$.
\end{proposition}
\begin{proof}
Let~$X$ be a~$\Z^d$-subshift,~$Y$ a~$\Z^e$-subshift,~$f:X\to Y$ a weak factor map coming with~$\varphi:\Z^e\to\Z^d$. If~$x\in X$ is strongly periodic, i.e.~its orbit~$O_x:=\{\sigma^p(x):p\in\Z^d\}$ is finite, then the orbit of~$f(x)$ is finite because it is contained in~$f(O_x)$: indeed,~$\sigma^p(f(x))=f(\sigma^{\varphi(p)}(x))\in f(O_x)$. As~$f$ is continuous, it sends a dense set to a dense set. 
\end{proof}
Our subshift~$X$ contains a dense set of strongly periodic configurations: given a valid~$S_n$-pattern~$\pi$, extend~$\pi$ on~$[0,n-1]\times [0,n]$ by adding a row filled with~$2$'s at the top, and copy this pattern at each point of the grid~$(n,0)\Z+(0,n+1)\Z$, yielding a strongly periodic configuration extending~$\pi$. However in~$\Zb$ the only strongly periodic configurations are its two fixed points, which are not dense.

We now show that~$X\notin\nar$, using a different but related argument.
\begin{proposition}
$X\notin \nar$.
\end{proposition}
\begin{proof}
We show that~$X$ is not the image of a narrow function and that~$X$ is weakly mixing, which implies that~$X\notin\nar$ by Proposition \ref{prop_wmix_narrow}.

Let~$f:A^\N\to X$ be an~$r$-narrow function. Let~$b_A(r)=|A|^\frac{r(r+1)}{2}$ from Proposition \ref{prop_zbar_narrow}. Let~$K\in\N$ be such that~$p:=\frac{b_A(r)}{K}\leq \frac{1}{2}$, and let~$N\in\N$ be such that~$|A|^{r(K+1)}p^N<1$. Consider the valid pattern~$\pi=200\ldots002$ on the region~$F=\{(0,0),\ldots,(0,K)\}$. Let~$\xi$ be a pattern on~$\W_f(F)$ that is sent to~$\pi$ under~$f$, and let~$f_\xi:A^\N\to X$ be obtained by first replacing the content of~$\W_f(F)$ in the input by~$\xi$ and then applying~$f$, and note that~$f_\xi$ is~$r$-narrow. Let~$i\in [1,N]$ and~$F_i=\{(i,0),\ldots,(i,K)\}$. In~$\im{f_\xi}$, there are at most~$b_A(r)$ ways of filling~$F_i$, which is a fraction~$p$ of the number of possible contents. Therefore, the fraction of the possible contents of the rectangle~$R:=F_1\cup\ldots\cup F_N$ reached by~$f_\xi$ is at most~$p^N$. As~$f$ is~$r$-narrow,~$|\W_f(F)|\leq r|F|=r(K+1)$, so there are at most~$|A|^{rK}$ input patterns~$\xi$ sent to~$\pi$, so a fraction of at most~$|A|^{r(K+1)}p^N<1$ contents of~$R$ that are reached by~$f$. Therefore,~$f$ is not surjective.

It is easy to see that~$X$ is weakly mixing by observing that~$S_m$ and~$(0,m+1)+S_m$ are independent. Indeed, if~$x,y$ are valid configurations, then the configuration~$z$ which coincides with~$x$ strictly under row~$m$ and with~$y$ strictly above row~$m$ and filled with~$2$'s on row~$m$ is valid. Therefore, any valid pattern on~$S_m$ and any valid pattern on~$(0,m+1)+S_m$ together form a valid pattern.
\end{proof}

%sss
\subsubsection{Irrational rotation}\label{sec_rotation}
We show that the classical~$\Z$-subshifts encoding irrational rotations on the circle do not belong to~$\nar$. It confirms the intuition that building a valid configuration cannot be achieved in local way, because all the cells need to share an infinite amount of information, namely the point of the circle that is encoded into the configuration.

We recall the definitions, more details can be found in Pytheas Fogg \cite[Section 6.1.2]{PF02}.

Let~$\R/\Z$ be the circle group. Let~$\alpha\in (0,1)$ be irrational and let~$R_\alpha:\R/\Z\to\R/\Z$ be the irrational rotation sending~$x$ to~$x+\alpha\mod 1$. Consider two partitions of the circle:
\begin{align*}
I_0&=[0,1-\alpha)&I'_0&=(0,1-\alpha]\\
I_1&=[1-\alpha,1)&I'_1&=(1-\alpha,1].
\end{align*}

To any~$x\in \R/\Z$ we associate the two sequences~$w(x),w'(x)\in\{0,1\}^\Z$ defined by~$w_n(x)=a$ such that~$R_\alpha^n(x)\in I_a$ and~$w'_n(x)=a$ such that~$R_\alpha^n(x)\in I'_a$. More directly, these sequences can be defined for~$x\in [0,1]$ by
\begin{align*}
w_n(x)&=\floor{(n+1)\alpha+x}-\floor{n\alpha+x},\\
w'_n(x)&=\ceil{(n+1)\alpha+x}-\ceil{n\alpha+x}.
\end{align*}

We then consider the subshift containing all these sequences:
\begin{equation*}
X_\alpha=\{w(x):x\in [0,1)\}\cup \{w'(x):x\in [0,1)\}.
\end{equation*}

%It is indeed shift-invariant as~$\sigma^k(w(x))=w(x+k\alpha)$ and similarly for~$w'$.

For most~$x$, one has~$w(x)=w'(x)$. The two sequence are different only if some~$n\alpha+x$ is an integer, in which case~$w_{n-1}(x)w_n(x)=10$ and~$w'_{n-1}(x)w'_n(x)=01$, and the two sequences coincide elsewhere.

%The mappings~$x\mapsto w(x)$ and~$x\mapsto w'(x)$ are injective. Moreover,~$x$ can be uniquely recovered from the restriction of~$w(x)$ or~$w'(x)$ to any half-line~$(-\infty,k]$ or~$[k,+\infty)$. Indeed, for any~$k\in\N$,~$\{n\alpha-\floor{n\alpha}:n\geq k\}$ is dense in~$[0,1]$ so, so if~$x<y$ then some arbitrarily far bit of~$w(x)$ and~$w(y)$ will differ; the same holds for~$n\leq -k$.

It turns out that~$X_\alpha$ is not weakly mixing, so we need another argument.
\begin{proposition}
$X_\alpha\notin\nar$.
\end{proposition}
\begin{proof}
Let~$F\subseteq\Z$ be a finite set, and~$u\in\{0,1\}^F$ a finite pattern. We show that~$X\cap [u]$ is not the image of a narrow function, which implies that~$X\notin \nar$ by Proposition \ref{prop_narrow}.

We say that~$x\in\R/\Z$ is \emph{generic} if no~$n\alpha+x$ is an integer. The generic points are dense in the circle, and their images by~$w$ and~$w'$ are dense in~$X_\alpha$. Let~$x$ be generic such that~$w(x)=w'(x)$ extends~$u$. For a real number~$z$, let~$\{z\}=z-\floor{z}\in [0,1)$ be its fractional part. Let~$i_0\in\N$ be such that~$1-\alpha-2^{-i_0-1}>0$, and~$\epsilon=\min \{\alpha,1-\alpha-2^{-i_0-1}\}$. For each~$i\geq i_0$, there exists~$n_i\in\Z$ such that~$\{n_i\alpha+x\}\in(1-\alpha-2^{-i},1-\alpha-2^{-i-1})$. Let~$v$ be a finite extension of~$u$ which is compatible with~$w(x)$ and such that if~$w(y)$ extends~$v$, then~$|x-y|<\epsilon$. Let~$g:X\cap[v]\to \{0,1\}^\Z$ send~$s$ to~$i\mapsto s(n_i)$. We claim that the image of~$g$ is an infinite set of non-decreasing binary sequences.

One has~$w_{n_i}(x)=0$, so~$\im{g}$ contains the zero sequence.

Let~$y$ be such that~$w(y)$ extends~$v$. We show that~$g(w(y))$ is non-decreasing. As~$|x-y|<\epsilon$, one has~$\{n_i\alpha+y\}=\{n_i\alpha+x\}+(y-x)$ because these quantities are equal modulo~$1$ and both belong to~$[0,1)$. If~$w_{n_i}(y)=1$, then~$\{n_i\alpha+y\}\geq 1-\alpha$, so~$\{n_{i+1}\alpha+y\}=\{n_{i+1}\alpha+x\}+(y-x)> \{n_i\alpha+x\}+(y-x)=\{n_i\alpha+y\}\geq 1-\alpha$, so~$w_{n_{i+1}}(y)=1$. Therefore,~$g(w(y))$ is indeed non-decreasing.

We show that~$g(X\cap [v])$ has infinitely many elements. As~$x$ is generic, the function~$w$ is continuous at~$x$ so there exists~$\delta>0$ such that if~$|x-y|<\delta$, then~$w(y)$ extends~$v$. Let~$i_1$ be such that~$2^{-i_1}<\delta$. For~$i\geq i_1$, let~$y=x+1-\alpha-2^{-i}-\{n_i\alpha+x\}$. One has~$|x-y|<\delta$ and~$\{n_i\alpha+y\}=1-\alpha-2^{-i}$. The sequence~$g(w(y))$ starts with~$i$~$0$'s, followed by~$1$'s.

The function~$g$ is~$1$-narrow, because its output at position~$i$ is determined by its input at position~$n_i$. Therefore, if~$X\cap [v]$ was the image of a narrow function then by composing with~$g$, we would obtain a narrow function whose image is an infinite set of non-decreasing sequences, i.e.~of elements of~$\Zb$, which is a contradiction. As~$[v]\subseteq [u]$,~$X\cap [u]$ is not the image of a narrow function by Proposition \ref{prop_narrow}.
\end{proof}

%SSSSSSS
\section{Subshifts that can be locally generated}\label{sec_positive}

We now present examples of subshifts that belong to~$\gen$. We first show that~$\gen$ is preserved under higher power presentations. This result will be used in the next arguments.

%ssss
\subsection{Higher power presentations and intertwining}\label{sec_block}
A common way of building a subshift from another one is to gather the cells into rectangular blocks and interpret the content of each block as an element of a larger alphabet. It can be found in \cite[Definition 1.4.4]{LM95} in the one-dimensional case, where it is called a higher power presentation of the subshift.

This operation sometimes results in a subshift which is easier to analyze. Intuitively, this operation should not affect the possibility of local generation. Indeed, we show that a higher power presentation of a subshift belongs to~$\gen$ if and only if the original subshift belongs to~$\gen$. Although intuitive, it turns out that this result is not so straightforward to prove.% We will use this result in Section \ref{sec_positive}, it can be skipped at first reading.

Let~$a=(a_1,\ldots,a_d)\in\Z^d$ be a vector with positive coordinates and consider the hyper-rectangle~$R_a=[0,a_1-1]\times\ldots \times [0,a_d-1]$.

\begin{definition}
To any~$\Z^d$-subshift~$X\subseteq \Sigma^{\Z^d}$ we associate its~$a$\textbf{-higher power presentation}~$[X]_a$, which is a~$\Z^d$-subshift on the alphabet~$\Gamma:=\Sigma^{R_a}$ defined as follows. Let~$F_a:\Sigma^{\Z^d}\to \Gamma^{\Z^d}$ be such that~$F_a(x)_p$ is the~$R_a$-pattern appearing at position~$ap=(a_1p_1,\ldots,a_dp_d)$ in~$x$. We then define~$[X]_a=\im{F_a}$.
\end{definition}

The subshift~$[X]_a$ is always a weak factor of~$X$, but not the other way round (see Example \ref{ex_block}]).

The main goal of this section is to show that higher power presentations behave well w.r.t.~the class~$\gen$, as expressed by the next result.

\begin{theorem}[\ti{$\gen$ and higher power presentations}]\label{thm_block_gen}
Let~$a\in\Z^d$ have positive coordinates and~$X$ be a~$\Z^d$-subshift. One has
\begin{equation*}
X\in\gen\iff [X]_a\in\gen.
\end{equation*}
\end{theorem}

The rest of this section is devoted to the proof of Theorem \ref{thm_block_gen} and can be skipped at first reading.

Let us start with a result highlighting the role of the compact subshift~$K$ in the definition of~$\gen$ when it is finite.
\begin{proposition}[\ti{Higher power presentations and factors}]\label{prop_block_factor}
Let~$X,Y$ be subshifts. The following conditions are equivalent:
\begin{itemize}
\item $Y$ is a factor of~$K\times X$ for some finite~$\Z^d$-set~$K$,
\item There exists~$a\in\Z^d$ with positive coordinates such that~$[Y]_a$ is a finite union of factors of~$[X]_a$.
\end{itemize}
\end{proposition}
\begin{proof}
Assume that~$Y$ is a factor of~$K\times X$ for some finite~$\Z^d$-set~$K$. For each~$k\in K$, let~$f_k:X\to Y$ be defined by~$f_k(x)=f(k,x)$. As~$K$ is finite, each vector~$e_i$ of the canonical basis of~$\Z^d$ has a multiple~$a_ie_i$, with~$a_i>0$, that acts as the identity on~$K$. Let~$a=(a_1,\ldots,a_d)$. For any~$p\in\Z^d$,~$ap$ acts as the identity on~$K$. It implies that~$\sigma^{ap}\circ f_k=f_k\circ \sigma^{ap}$:
\begin{align*}
\sigma^{ap}\circ f_k(x)&=\sigma^{ap}\circ f(k,x)\\
&=f(ap\cdot k,\sigma^{ap}(x))\\
&=f(k,\sigma^{ap}(x))\\
&=f_k\circ \sigma^{ap}(x).
\end{align*}
Let~$F_X:X\to [X]_a$ and~$F_Y:Y\to [Y]_a$ be the canonical homeomorphisms, which satisfy~$\sigma^p\circ F_Z=F_Z\circ \sigma^{ap}$ (for~$Z=X$ or~$Y$).  Each function~$F_Y\circ f_k\circ F_X^{-1}$ is a factor map from~$[X]_a$ to its image, and~$[Y]_a$ is the finite union over~$k\in K$ of these images.

Conversely, assume that~$[Y]_a=\bigcup_{i\in I}f_i([X]_a)$ where~$I$ is finite,~$f_i:[X]_a\to [Y]_a$ is continuous and commutes with the shift actions. Let~$K=I\times \Z^d/a\Z^d$ where~$I$ is endowed with the trivial action~$p\cdot i=i$. The functions~$f'_i=F_Y^{-1}\circ f_i\circ F_X:X\to Y$ satisfy~$\sigma^{ap}\circ f'_i\circ \sigma^{-ap}=f'_i$, which implies that for any~$p\in\Z^d$,~$\sigma^p\circ f'_i\circ \sigma^{-p}$ only depends on the equivalence class~$[p]$ of~$p$ modulo~$a\Z^d$. Therefore, the following function~$g:K\times X\to Y$ is well-defined:
\begin{equation*}
g(i,[p],x)=\sigma^p\circ f'_i\circ \sigma^{-p}(x).
\end{equation*}
It is continuous and surjective; moreover, it commutes with the actions:
\begin{align*}
\sigma^q\circ g(i,[p],x)&=\sigma^q\circ \sigma^p\circ f'_i\circ \sigma^{-p}(x)\\
&=\sigma^{p+q}\circ f'_i\circ \sigma^{-(p+q)}(\sigma^q(x))\\
&=g(i,[p+q],\sigma^q(x))\\
&=g(q\cdot (i,[p],x)).
\end{align*}
Therefore,~$g$ is a factor map from~$K\times X$ to~$Y$.
\end{proof}

We introduce another way of building new subshifts from old ones, that will be used to prove Theorem \ref{thm_block_gen}.

\begin{definition}[\ti{Intertwining}]\label{def_twin}
To a subshift~$X\subseteq A^{\Z^d}$ we associate the subshift~$\twin{X}$ whose configurations are made of~$|R_a|$ intertwined configurations of~$X$. Precisely, let~$\psi_a:(A^{\Z^d})^{R_a}\to A^{\Z^d}$ be defined by~$\psi_a(\overline{x})_{aq+r}=x^r_q$, where~$r\in R_a$ and~$\overline{x}=(x^r)_{r\in R_a}$. We define~$\twin{X}=\psi_a(X^R_a)$.
\end{definition}

One easily checks that~$\twin{X}$ is indeed a subshift. 
%Let us check that~$\twin{X}$ is indeed a subshift. It is compact because~$\psi_a$ is a homeomorphism so~$\psi_a(X)$ is compact; it is shift-invariant because one can check that~$\sigma^{p}\circ \psi_a(\overline{x})=\psi_a(\overline{y})$ with~$y^r=\sigma^{q+q'}(x^{r'})$ where~$p=aq+s$,~$r+s=aq'+r'$ and~$r',s\in R_a$. It will be useful to consider the map that sends~$\overline{x}$ to~$\overline{y}$, which is~$\phi_p:(A^{\Z^d})^{R_a}\to (A^{\Z^d})^{R_a}$ defined by~$\phi_p=\psi_a^{-1}\circ \sigma^p\circ \psi_a$.
The restriction~$\psi_a:X^{R_a}\to \twin{X}$ is a homeomorphism but is not a conjugacy. One has~$\sigma^{ap}\circ \psi_a=\psi_a\circ \sigma^p$, so~$\psi_a^{-1}:\twin{X}\to X^{R_a}$ is a weak factor map.

\begin{proposition}[\ti{$\gen$ and intertwining}]\label{prop_twin_gen}
$X\in\gen\iff \twin{X}\in\gen$.
\end{proposition}
\begin{proof}
As~$X$ is a factor of~$X^{R_a}$ which is a weak factor of~$\twin{X}$, if~$\twin{X}\in\gen$ then~$X\in\gen$.

For the converse implication, it is not difficult to see that intertwining preserves basic subshifts and commutes with finite products and factor maps:
\begin{itemize}
\item If~$K$ is a countable subshift, then so is~$\twin{K}$, because it is homeomorphic to~$K^{R_a}$ which is countable,
\item If~$X=A^{\Z^d}$ is the fullshift, then~$\twin{X}=A^{\Z^d}$ is the fullshift,% Indeed one has~$X^{R_a}=(A^{\Z^d})^{R_a}$ and~$\psi_a$ is surjective, so~$\psi_A(X^{R_a})=A^{\Z^d}$.
\item If~$X=A^{\Z^d/H}$ is a periodic shift, then~$\twin{X}=A^{\Z^d/aH}$ is a periodic shift as well,% Let~$\overline{x}\in (A^{\Z^d})^{R_a}$ and~$y=\psi_a(\overline{x})$. For each~$h\in H$,~$y$ is~$ah$-periodic if and only if each~$x^r$ is~$h$-periodic: indeed,~$y_{ap+r}=x^r_p$ and~$y_{(ap+r)+ah}=y_{a(p+h)+r}=x^r_{p+h}$. As~$\psi_a$ is surjective,~$\psi_a(X^{R_a})$ is the set of~$aH$-periodic configurations.
\item $\twin{X\times Y}$ is conjugate to~$\twin{X}\times \twin{Y}$,% Indeed, there is a natural conjugacy~$f:(X\times Y)^{R_a}\to X^{R_a}\times Y^{R_a}$, and the function~$(\psi_a,\psi_a)\circ f\circ \psi_a^{-1}:\twin{X\times Y}\to\twin{X}\times \twin{Y}$ is a conjugacy as well. 
\item Finally, if~$Y$ is a factor of~$X$, then~$\twin{Y}$ is a factor of~$\twin{X}$.% Let~$f:X\to Y$ be a factor map. Let~$f^{R_a}:X^{R_a}\to Y^{R_a}$ and~$\twin{f}=\psi_a\circ f^{R_a}\circ \psi_a^{-1}$. Note that the function~$\phi_p:(A^{\Z^d})^{R_a}\to (A^{\Z^d})^{R_a}$ defined after Definition \ref{def_twin} shifts and permutes the components of its input, so it commutes with~$f^{R_a}$. As~$\sigma^p\circ \psi_a=\psi_a\circ \phi_p$, one has~$\sigma^{p}\circ \twin{f}=\twin{f}\circ \sigma^{p}$.
\end{itemize}

As a result, if~$X\in\gen$, i.e.~$X$ is a factor of a finite product of basic subshifts~$B_i$, then~$\twin{X}$ is a factor of the product of basic subshifts~$\twin{B_i}$, so~$\twin{X}\in\gen$.
\end{proof}

Intertwining and higher power presentations are related in the following way.
\begin{proposition}\label{prop_inter_block}
$X$ is a factor of~$[\twin{X}]_a$.
\end{proposition}
\begin{proof}
Let~$\pi:X^{R_a}\to X$ be the projection sending~$\overline{x}$ to~$x^0$. It is a factor map. Let~$f=\pi\circ \psi_a^{-1}:\twin{X}\to X$. It is continuous and surjective, and satisfies~$\sigma^p\circ f=f\circ \sigma^{ap}$. As a result,~$f$ is a factor map from~$[\twin{X}]_a$ to~$X$.
\end{proof}

We now have all the ingredients to prove the main result of this section.
\begin{proof}[Proof of Theorem \ref{thm_block_gen}]
The function~$F_a$ satisfies~$\sigma^p\circ F_a=F_a\circ \sigma^{ap}$, so~$[X]_a$ is a weak factor of~$X$ via~$F_a$ and~$\varphi(p)=ap$. In particular, if~$X\in\gen$ then~$[X]_a\in\gen$. Our goal is to show that the converse also holds.

Conversely, if~$[X]_a\in\gen$, then~$[X]_a$ is a factor of a finite product~$P$ of basic subshifts. In turn,~$P$ is a factor of~$[\twin{P}]_a$ by Proposition \ref{prop_inter_block}. Therefore,~$[X]_a$ is a factor of~$[\twin{P}]_a$, which implies that~$X$ is a factor of~$K\times \twin{P}$ for some finite~$\Z^d$-set~$K$ by Proposition \ref{prop_block_factor}. As~$\twin{P}$ belongs to~$\gen$ ($\twin{P}$ is even a finite product of basic subshifts) by Proposition \ref{prop_twin_gen}, so do~$K\times \twin{P}$ and~$X\in\gen$.
\end{proof}

The case of~$\nar$ is much easier to prove.
\begin{proposition}[\ti{$\nar$, higher power presentations and intertwining}]\label{prop_block_nar}
Let~$a\in\Z^d$ have positive coordinates and~$X$ be a~$\Z^d$-subshift. One has
\begin{equation*}
X\in\nar\iff [X]_a\in\nar\iff \twin{X}\in\nar.
\end{equation*}
\end{proposition}
\begin{proof}
The canonical homeomorphisms~$F_a:X\to [X]_a$ and~$\psi_a:X^{R_a}\to\twin{X}$, as well as their inverses, are narrow functions, so they preserve the property of being in~$\nar$. Note that~$X\in\nar\iff X^{R_a}\in\nar$ because any finite product of narrow functions is narrow, and any projection is narrow.
\end{proof}
%
%\begin{proposition}
%The class~$\gen$ contains a subshift which is not an SFT.
%\end{proposition}
%\begin{proof}
%Let~$X\subseteq \{0,1\}^\Z$ be the smallest subshift containing the sequences~$x$ satisfying~$x_k=0$ if~$k\notin\{\pm 2^n:n\in\N\}$. Then~$X\in \gen$. Let~$K$ be the smallest subshift containing the sequences~$x$ satisfying~$x_k=0$ iff~$k\notin\{\pm 2^n:n\in\N\}$. It is countable, and the function~$f:K\times\{0,1\}^\Z$ outputing the bitwise AND operation is a factor map.
%
%Let~$f:\{0,1\}^\Z\to \{0,1\}^\Z$ send~$x$ to~$y$ defined by~$y_k=x_k$ if~$k=\pm 2^n$ for some~$n$,~$y_k=0$ otherwise. $f$ is a local function, because~$k$ is determined by~$\{k\}=k\cdot R$ where~$R=\{0\}$.
%
%$X$ is the union of the images of~$\sigma^p\circ f\circ \sigma^{-p}$, which are all local functions, so~$X\in \gen$.
%\end{proof}

%ssss
\subsection{One-dimensional SFTs}
One-dimensional SFTs are known to be much simpler than higher-dimensional ones in many respects: several problems such as non-emptyness are decidable, many properties have simple characterizations that do not extend to higher-dimensions, etc. We show that they are also simpler in the sense of local generation.
\begin{theorem}\label{thm_1D_SFT}
Every one-dimensional SFT belongs to~$\gen$.
\end{theorem}
To a finite directed graph~$G=(V,E)$ (possibly with loops) we associate the subshift~$X_G$, called vertex shift in \cite[Definition 2.3.7]{LM95}, which is the set of bi-infinite walks through~$G$, i.e.~$X_G=\{x\in V^\Z:\forall n, (x_n,x_{n+1})\in E\}$.

It is well-known that every SFT~$X$ is conjugate to~$X_G$ for some finite directed graph~$G$ \cite[Proposition 2.3.9 (3)]{LM95}, obtained as follows. Let~$n$ be larger than the lengths of forbidden patterns, let~$V$ be the set of valid words of length~$n$ and~$E$ be the set of pairs~$(au,ub)$ where~$|u|=n-1$,~$a,b$ are letters and~$au,ub\in V$. The function~$f:X\to X_G$ such that~$f(x)_i=x_i\ldots x_{i+n-1}$ is then a conjugacy.

We show that for any finite directed graph~$G$,~$X_G\in\gen$. The idea is that the walks through~$G$ can be described in the following way: for some~$k\geq 1$ and some equivalence relation between the vertices of~$G$, a walk is described by giving, for each~$n$ which is a multiple of~$k$, the equivalence class of the vertex visited at position~$n$ (we call this sequence a partial sequence) and then choosing a vertex in each equivalence class and filling the intermediate positions. The careful choice of~$k$ and the equivalence relation makes the set of partial sequences countable. The second step is local, because the choices of vertices between positions~$n$ and~$n+k$ can be made by knowing the equivalence classes at positions~$n$ and~$n+k$ only. Let us now give the details.

To a finite directed graph~$G=(V,E)$ we associate its powers~$G^k=(V,E^k)$, where~$(u,v)\in E^k$ if there exists a walk of length~$k$ from~$u$ to~$v$ in~$G$, i.e.~there exist vertices~$v_0,\ldots,v_k$ with~$v_0=u, v_k=v$ and~$(v_i,v_{i+1})\in E$ for all~$i<k$.
\begin{lemma}\label{lem_transitive}
Let~$G$ be a finite directed graph (we allow loops). There exists~$k\geq 1$ such that~$G^k$ is transitive (i.e., if there is a walk from~$u$ to~$v$, then there is an edge from~$u$ to~$v$).
\end{lemma}
\begin{proof}
We first show the result if the graph has the property that every vertex that belongs to a circuit has a loop to itself (call it the loop property).

Let~$u,v$ be two distinct vertices and let~$L_{u,v}\subseteq\N$ be the set of lengths of walks from~$u$ to~$v$. We claim that~$L_{u,v}$ is either finite or co-finite. Indeed, if there is no walk from~$u$ to~$v$ visiting some vertex twice, then the lengths of these walks are bounded by the number of vertices; if there is a path from~$u$ to~$v$ visiting some vertex~$w$ twice, then by the loop property~$w$ has a loop to itself. The portion of the walk between the two occurrences of~$w$ can be replaced by that loop and one can create walks of any larger length by cycling though that loop as any times as wanted.

Let~$k_{u,v}$ be an upper bound on~$L_{u,v}$ or its complement, and let~$k=\max_{u,v} k_{u,v}$. Let~$u\neq v$. If there is a walk in~$G^k$ from~$u$ to~$v$, then there is a walk in~$G$ whose length is a non-zero multiple of~$k$, in particular this length is at least~$k$. By choice of~$k$, there is a walk in~$G$ of length exactly~$k$, i.e.~there is an edge in~$G^k$ from~$u$ to~$v$. Therefore,~$G^k$ is transitive.

We now prove that for any finite directed graph~$G$, there exists~$k\geq 1$ such that~$G^k$ has the loop property. For each vertex~$u$ that belongs to a circuit, let~$k_u$ be the length of such a circuit. Let~$k$ be a common multiple of the~$k_u$'s. In~$G^k$, if a vertex~$u$ belongs to a circuit, then~$u$ belongs to a circuit in~$G$ so it belongs to a circuit in~$G$ of length~$k$, hence~$u$ has a loop to itself in~$G^k$. Therefore,~$G^k$ has the loop property. By the first part of the proof,~$G^{kl}$ is transitive for some~$l$. 
\end{proof}
Note that~$G^k$ need not be the transitive closure of~$G$: for instance if~$G$ is a cycle of length~$2$, then~$k$ must be a multiple of~$2$ and~$G^k$ has two connected components, each one consisting of a vertex with a loop.

\begin{lemma}\label{lem_transitive_gen}
Let~$G$ be a finite directed graph which is transitive. Then~$X_G\in\gen$.
\end{lemma}
\begin{proof}
Let~$H$ be the condensation of~$G$, which is the directed acyclic graph whose vertices are the strongly connected components of~$G$, but in which we keep the loops of~$G$. More precisely, if~$C$ is a strongly connected component of~$G$, then either~$C$ is a singleton with no loop, or each vertex in~$C$ has a loop to itself. In~$H$,~$C$ will have a loop to itself it and only if the vertices of~$C$ have loops to themselves in~$G$. As~$H$ has no cycle other than loops, the subshift~$X_H$ is countable: a configuration is obtained by concatenating a finite number of constant blocks (either there is one bi-infinite block, or the starting and ending blocks are infinite on one side and the intermediate blocks are finite), so it can be described by the positions of the blocks and the constant symbol in each block.

We define a factor map~$F:X_H\times V^\Z\to X_G$. First, for each strongly connected component~$C\subseteq V$ of~$G$, fix a surjective function~$f_C:V\to C$ which is the identity on~$C$. Given~$c\in X_H$ and~$s\in V^\Z$, let~$y=F(c,s)$ be defined by~$y_n=f_{c_n}(s_n)$ for all~$n\in\Z$.

First,~$y\in X_G$. One has~$(c_n,c_{n+1})\in E_H$ and~$y_n\in c_n$ and~$y_{n+1}\in c_{n+1}$ so there is a walk from~$y_n$ to~$y_{n+1}$ in~$G$. As~$G$ is transitive,~$(y_n,y_{n+1})$ is an edge in~$G$.

Every~$y\in X_G$ belongs to~$\im{F}$: let~$c_n$ be the strongly connected component of~$y_n$ and let~$s=y$. One has~$c\in X_H$ and~$f_{c_n}(y_n)=y_n$ because~$y_n\in c_n$, so~$F(c,s)=y$.

Therefore, the function~$F$ is surjective. It is continuous and commutes with the shift action, because~$F(c,s)_n$ is determined by~$c_n$ and~$s_n$ by a rule that does not depend on~$n$. Therefore,~$X_G$ is a factor of the product of a countable subshift and a fullshift, so~$X_G\in\gen$.
\end{proof}

\begin{proof}[Proof of Theorem \ref{thm_1D_SFT}]
Let~$G$ be a finite directed graph. By Lemma \ref{lem_transitive}, there exists~$k$ such that~$G^k$ is transitive. By Lemma \ref{lem_transitive_gen},~$X_{G^k}\in\gen$. We show that the~$k$-power presentation~$[X_G]_k$ is a factor of the product of~$X_{G^k}$ with a fullshift. Let~$\Sigma$ be the alphabet of~$[X_G]_k$, which is the set of walks~$(v_0,\ldots,v_{k-1})$ in~$G$. For each pair of vertices~$u,v$, let~$\Sigma_{u,v}\subseteq\Sigma$ be the set of walks~$(v_0,\ldots,v_{k-1})$ such that~$v_0=u$ and~$(v_{k-1},v)\in E$. When~$\Sigma_{u,v}\neq\emptyset$, let~$f_{u,v}:\Sigma\to\Sigma_{u,v}$ be a surjective function which is the identity on~$\Sigma_{u,v}$. We define a factor map~$F:X_{G^k}\times \Sigma^Z\to [X_G]_k$. Given~$x\in X_{G^k}$ and~$s\in\Sigma^\Z$, let~$y=F(x,s)$ be defined by~$y_n=f_{x_n,x_{n+1}}(s_n)$.

First, it is well-defined and~$y\in [X_G]_k$. Indeed,~$(x_n,x_{n+1})$ is an edge in~$G^k$ so there is a walk of length~$k$ from~$x_n$ to~$x_{n+1}$ in~$G$, and~$f_{x_n,x_{n+1}}(s_n)$ is such a walk (without its last vertex).

Every~$y\in [X_G]_k$ belongs to~$\im{F}$. Let~$x_n$ be the starting vertex of~$y_n$ and~$s_n=y_n$. One has~$f_{x_n,x_{n+1}}(s_n)=y_n$ so~$F(x,s)=y$.

The function~$F$ is therefore surjective. It is continuous and commutes with the shift action, because~$F(x,s)_n$ is determined by~$x_n,x_{n+1},s_n$ by a rule that does not depend on~$n$. Therefore,~$[X_G]_k$ is a factor of~$X_{G^k}\times \Sigma^Z$. As~$X_{G^k}$ belongs to~$\gen$, so do~$[X_G]_k$ and~$X_G$ by Theorem \ref{thm_block_gen}.
\end{proof}

%ssss
\subsection{Bicolor Wang tilesets}
We now present two tilesets that belong to~$\gen$.

%sssssss
\subsubsection{The wires}\label{sec_all_even}
We start with the tileset~$T$ that was mentioned in the introduction, shown in Figure \ref{fig_all_even}, and give the details.
\begin{figure}[!ht]
\centering
\includegraphics{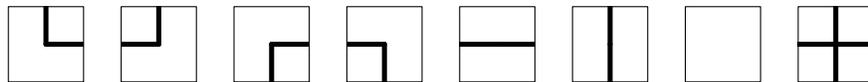}
\caption{The wires}\label{fig_all_even}
\end{figure}

The induced subshift~$X_T$ belongs to~$\gen$ because it is a factor of the full-shift~$\{0,1\}^{\Z^2}$. The factor map is given by a rule determining the tile at position~$(i,j)$ from the 4 bits at positions~$(i,j),(i+1,j),(i,j+1),(i+1,j+1)$, which can be thought as lying on the corners of the tile. The rule assigns to each edge the sum modulo $2$ of the bits lying at the endpoints of the edge (see Figure \ref{fig_all_even_rule}).
\begin{figure}[!ht]
\centering
\includegraphics{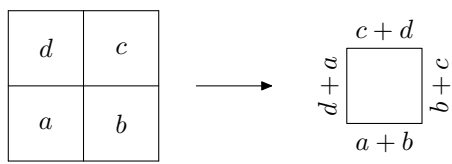}
\caption[Local rule]{The rule determined a tile from $4$ bits at its corners}\label{fig_all_even_rule}
\end{figure}

One easily checks that:
\begin{itemize}
\item The tiles produced by this rule are even: indeed, given bits~$a,b,c,d$, the sum~$(a+b)+(b+c)+(c+d)+(d+a)$ is even,
\item The tiles match: if two cells are adjacent, then the rule applied to each cell will assign the same color to their common edge, because it only depends on the bits lying at the endpoints of the edge,
\item All the valid configurations are reached: given a valid configuration~$y$, let~$x\in\{0,1\}^{\Z^d}$ assign to~$(i,j)$ the parity of the number of black edges that are met on an arbitrary path from the origin to~$(i,j)$ (the points of~$\Z^d$ are the corners of the tiles, and the path is a finite sequence of edges). This process does not depend on the path: as the tiles are even, any closed loop meets an even number of black edges, so two path starting and ending at the same points have the same parity. Applying the rule to~$x$ will produce~$y$: let~$e$ be an edge between vertices~$u$ and~$v$, let~$p_u$ be a path from the origin to~$u$ and~$p_v$ be the path from the origin to~$v$ obtained by adding~$e$ at the end of~$p_u$; the color of~$e$ in~$y$ is indeed the sum modulo~$2$ of the parities of~$p_u$ and~$p_v$. 
\end{itemize}
%
%For~$N,S,E,W\in\{0,1\}$, let~$t(N,S,E,W)$ be the tile with colors~$N,S,E,W$ on the north, south, east and west edges respectively. Note that~$T$ is the set of all the \emph{even} tiles, i.e.~$T=\{t(N,S,E,W):N+S+E+W=0\mod 2\}$.
%
%An
%
%
%The tile at position~$(i,j)$ is determined by the 4 bits at positions~$(i,j),(i+1,j),(i,j+1),(i+1,j+1)$, which can be thought as lying on the corners of the square. An edge is colored black if and only if its two endpoints have different bits. Every tile obtained this way is even, because its parity is twice the sum of the 4 bits lying at its corners, and the tiles obviously match. Every valid configuration~$y$ is reached: let~$x\in\{0,1\}^{\Z^2}$ associate to each point~$p\in \Z^2$ the parity of the number of wires intersecting any path from~$(0,0)$ to~$p$.
%
%
%Formally, the factor map is
%\begin{gather*}
%f:\{0,1\}^{\Z^2}\to X\\
%f(x)_{i,j}=t(x_{i,j+1}+x_{i+1,j+1},x_{i+1,j}+x_{i,j},x_{i+1,j}+x_{i+1,j+1},x_{i,j+1}+x_{i,j}).
%\end{gather*}

Figure \ref{fig_random_all_even} shows a configuration obtained by applying this procedure starting from a random grid of bits.
\begin{figure}[!ht]
\centering
\includegraphics{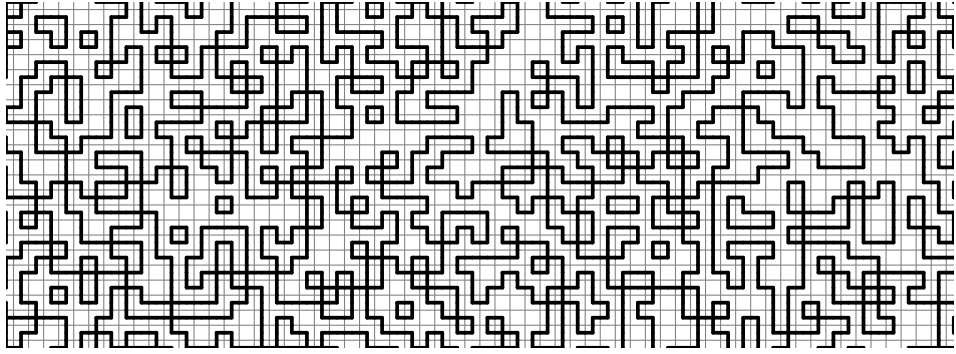}
\caption[Random tiling]{A random tiling with the wires}\label{fig_random_all_even}
\end{figure}

%ssssss
\subsubsection{Corners}\label{sec_corners}
The tileset shown in Figure \ref{fig_corners} belongs to~$\gen$.
\begin{figure}[!ht]
\centering
\includegraphics{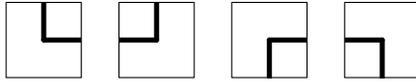}
\caption{The corners}\label{fig_corners}
\end{figure}

This tileset is the set of all the bicolor tiles whose opposite edges have opposite colors. It implies that for each pair of adjacent edges, all the colorations of these edges are allowed and each coloration uniquely determines the colors of the other two edges. Therefore, one can generate a valid configuration by assigning arbitrary colors to the edges lying on the $x$-axis and $y$-axis, and this assignment uniquely determines the tiles in all the cells. Moreover, the tile at position~$(i,j)$ can be easily obtained in a local way, by reading the colors of edge~$i$ on the~$x$-axis and the color of edge~$j$ on the vertical axis (see Figure \ref{fig_corners_rule}).%: for each~$(i,j)\in\Z^d$, the color of the edge from~$(i,j)$ to~$(i+1,j)$ is the sum modulo $2$ of~$j$ and the color of the edge from~$(i,0)$ to~$(i+1,0)$; the color of the edge from~$(i,j)$ to~$(i,j+1)$ is the sum modulo $2$ of~$i$ and the color of the edge from~$(0,j)$ to~$(0,j+1)$.
\begin{figure}[!ht]
\centering
\includegraphics{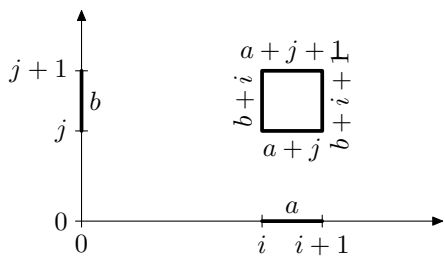}
\caption[Local rule]{The rule determining a tile from the colors on the edges on the $x$-axis and $y$-axis. The numbers~$i,j\in\Z$ are the coordinates, and~$a,b\in\{0,1\}$ are the colors. The colors of the edges are understood modulo 2.}\label{fig_corners_rule}
\end{figure}

%
%If one is given two sequences of bits~$a=(a_i)_{i\in\Z}$ and~$b=(b_j)_{j\in\Z}$ as inputs, then the rule determining the tile at position~$(i,j)$ only reads~$a_i$ and~$b_j$. One can see~$a$ as an element of~$\{0,1\}^{\Z^2/(1,0)\Z}$,~$b$ as an element of~$\{0,1\}^{\Z^2/(0,1)\Z}$ and the procedure as a surjective function~$f:\{0,1\}^{\Z^2/(1,0)\Z}\times \{0,1\}^{\Z^2/(0,1)\Z}\to X$. The input space of~$f$ can be expressed as~$\{0,1\}^E$ with~$E={\Z^2/(1,0)\Z}\sqcup {\Z^2/(0,1)\Z}$. The function~$f$ is local, because the output content of~$p$ is determined by~$p\cdot R$, where~$R$ contains two elements, the zeroes of~$\Z^2/(1,0)\Z$ and of~$\Z^2/(0,1)\Z$.
%
% The rule at cell~$(i,j)$ depends on the parities of~$i$ and~$j$, so~$f$ is strongly periodic of period~$2$. As a result,~$X$ belongs to~$\gen$ by Proposition \ref{prop_local_periodic}, and is the image of the factor map~$F:\Z^2/2\Z^2\times\{0,1\}^E\to X$ sending~$(p,x)$ to~$\sigma^p\circ f\circ\sigma^{-p}(x)$.

In more concrete terms,~$X_T$ is a factor of
\begin{equation*}
\Z^2/2\Z^2\times \{0,1\}^{\Z^2/(1,0)\Z}\times \{0,1\}^{\Z^2/(0,1)\Z}
\end{equation*}
and the factor map~$f$ can be expressed as follows. For~$a,b\in \Z/2\Z$, let~$t(a,b)\in T$ be the unique tile with colors~$a$ and~$b$ on its lower and left edges respectively. For~$(k,l)\in \Z^2/2\Z^2$,~$a\in \{0,1\}^{\Z^2/(1,0)\Z}$ and~$b\in \{0,1\}^{\Z^2/(0,1)\Z}$, the tile at position~$(i,j)$ in~$f((k,l),a,b)$ is
\begin{equation*}
t(a_i+l+j,b_j+k+i).
\end{equation*}

%
%Let~$K=\Z/2\Z$ with the~$\Z^2$-action~$(k,l)\cdot b=b+k+l\mod 2$ ($K$ can be seen as~$\Z^2/H$ with~$H=\{(x,y)\in\Z^2:x+y\mod 2=0\}\subseteq\Z^2$).
%
%\begin{equation*}
%f:K\times \{0,1\}^{\Z^2/(0,1)\Z}\times \{0,1\}^{\Z^2/(1,0)\Z}
%\end{equation*}
%be defined by~$f(b,x,y)_{i,j}=t(x(i)+i+j+b,y(j)+i+j+b)$.
%
%One has
%\begin{align*}
%(\sigma^{(k,l)}\circ f(b,x,y))_{i,j}&=f(b,x,y)_{k+i,l+j}\\
%&=t(x(k+i)+k+i+l+j+b,y(l+j)+k+i+l+j+b)\\
%%&=t(x(k+i)+a+k+b+l+i+j,\sigma^k(x),\sigma^l(y))_{i,j}\\
%&=f(b+k+l,\sigma^k(x),\sigma^l(y))_{i,j}\\
%&=f((k,l)\cdot(b,x,y))_{i,j}
%\end{align*}
%so~$f$ is indeed a factor map.

Figure \ref{fig_random_corners} shows a configuration obtained by applying this procedure to a pair of random bit sequences.
\begin{figure}[!ht]
\centering
\includegraphics{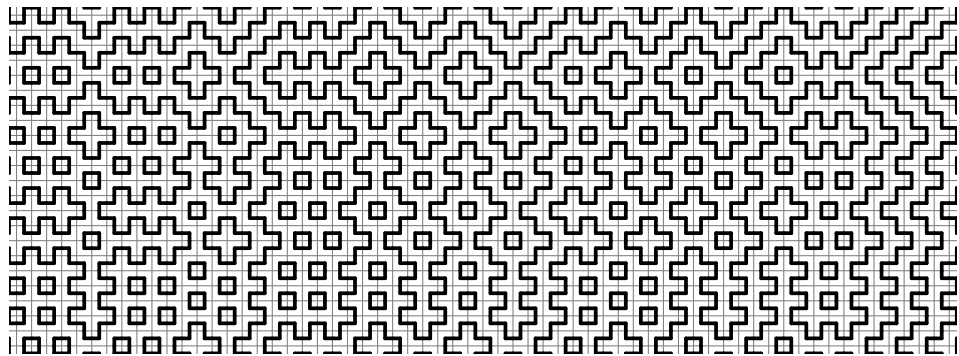}
\caption[Random tiling]{A random tiling with the corners}\label{fig_random_corners}
\end{figure}

%SSS
\section{Classification of the even bicolor tilesets}\label{sec_classification}
In a forthcoming article \cite{FH24b}, we classify all the \textbf{even bicolor tilesets}, which are the Wang tilesets with two colors in which each tile has an even number of each color. We summarize the classification below, all the details are presented in \cite{FH24b}.

There are 8 even bicolor tiles, which we draw using wires instead of colors for aesthetic reasons: \smalltile{1100}, \smalltile{0110}, \smalltile{0011}, \smalltile{1001}, \smalltile{1010}, \smalltile{0101}, \smalltile{0000}, \smalltile{1111}. There are~$2^8-1=255$ non-empty even bicolor tilesets. However, many tilesets are symmetric versions of each other, implying that the induced subshifts are weakly conjugate, so we only need to consider one tileset for each equivalence class. It turns out that there are 36 equivalence classes. Moreover, 8 of them are not minimal in the sense that some tile cannot be used in any tiling, so these tilesets induce the same subshift as a smaller tileset (which is possibly empty). All in all, there are 28 remaining subshifts:
\begin{itemize}
\item 13 subshifts belong to $\gen$ (2 of them were presented in Sections \ref{sec_all_even} and \ref{sec_corners}),
\item 8 subshifts do not belong to $\gen$; we expect that they do not belong to $\nar$, but leave the question open,
\item 7 subshifts do not belong to $\nar$.
\end{itemize}

Table \ref{table_even} shows a tileset for each one of these 28 equivalence classes. Table \ref{table_even_stupid} shows the 8 equivalence classes of non-minimal tilesets, with one tileset per class (they induce the empty subshift or the same subshift as a smaller tileset already given in Table \ref{table_even}).

\begin{table}[!ht]
\centering
\begin{tabular}{|l|l|l|}
\hline
\rule{0mm}{2.5ex}In $\gen$&Not in $\gen$ (not in $\nar$?)&Not in $\nar$\\
\hline\hline
\rule{0mm}{3.2ex}\includegraphics{even_classes/small_class-1-0}&\includegraphics{even_classes/small_class-14-0}&\includegraphics{even_classes/small_class-19-4}\\
\includegraphics{even_classes/small_class-16-0}&\includegraphics{even_classes/small_class-23-0}&\includegraphics{even_classes/small_class-9-1}\\
\includegraphics{even_classes/small_class-4-0}&\includegraphics{even_classes/small_class-15-0}&\includegraphics{even_classes/small_class-20-6}\\
\includegraphics{even_classes/small_class-10-0}&\includegraphics{even_classes/small_class-27-0}&\includegraphics{even_classes/small_class-32-1}\\
\includegraphics{even_classes/small_class-17-0}&\includegraphics{even_classes/small_class-34-0}&\includegraphics{even_classes/small_class-21-8}\\
\includegraphics{even_classes/small_class-18-4}&\includegraphics{even_classes/small_class-25-0}&\includegraphics{even_classes/small_class-24-1}\\
\includegraphics{even_classes/small_class-11-0}&\includegraphics{even_classes/small_class-29-1}&\includegraphics{even_classes/small_class-33-2}\\
\includegraphics{even_classes/small_class-31-0}&\includegraphics{even_classes/small_class-35-0}&\\\cline{2-3}
\multicolumn{3}{|l|}{\includegraphics{even_classes/small_class-22-0}}\\
\multicolumn{3}{|l|}{\includegraphics{even_classes/small_class-26-0}}\\
\multicolumn{3}{|l|}{\includegraphics{even_classes/small_class-28-0}}\\
\multicolumn{3}{|l|}{\includegraphics{even_classes/small_class-30-0}}\\
\multicolumn{3}{|l|}{\includegraphics{even_classes/small_class-36-0}}\\
\hline
\end{tabular}
\caption{Classification of the even bicolor tilesets}\label{table_even}
\end{table}

%For the subshifts that do not belong to~$\gen$, we do not know whether they belong to~$\nar$, although we expect that it is not the case.

%Some of the tilesets are not interesting because either they do not admit any tiling, i.e.~they induce the empty subshift, or one of the tiles cannot be used in a tileset, so they induce the same subshift as one of the tilesets from Table \ref{table_even}.
\begin{table}[!ht]
\centering
\begin{tabular}{|l|l|}
\hline
\rule{0mm}{2.5ex}Empty subshift&Unused tile\\
\hline\hline
\rule{0mm}{3.2ex}\includegraphics{even_classes/small_class-2-1}&\includegraphics{even_classes/small_class-3-4}\\
\includegraphics{even_classes/small_class-6-0}&\includegraphics{even_classes/small_class-5-1}\\
&\includegraphics{even_classes/small_class-7-0}\\
&\includegraphics{even_classes/small_class-12-0}\\
&\includegraphics{even_classes/small_class-8-0}\\
&\includegraphics{even_classes/small_class-13-0}\\
\hline
\end{tabular}
\caption{The non-minimal even bicolor tilesets}\label{table_even_stupid}
\end{table}

%SSS
\section{Future directions}\label{sec_future}

The main problem that is left open is to settle whether the triangles and domino subshifts from Sections \ref{sec_triangles} and \ref{sec_dominos} belong to~$\nar$. We do not expect so.

The main techniques developed in this article, Theorems \ref{thm_obs_C0} and \ref{thm_notin_nar} apply to weakly mixing subshifts only. A natural problem is to investigate subshifts that are not weakly mixing.

The definitions of~$\gen$ and~$\nar$ are first attempts to capture an intuitive idea, but other definitions of local generation are probably possible. For instance, in Proposition \ref{prop_ZbarZ} we considered the two-dimensional SFT~$X$ on the alphabet~$\{\whitetile,\blacktile\}$ with forbidden pattern \raisebox{-.5mm}{\includegraphics{Tiles/tiles-6}}, and we showed that~$X\notin \nar$. However, one can generate any configuration of~$X$ by the following procedure: first assign an element of~$\Zb$ to each row, and then fill each cell according to the element assigned to its row. This procedure is a function~$h:\Zb^\Z\to X$ such that each output cell is determined by \emph{one} input cell, so it is like a~$1$-narrow function, except that the input alphabet~$\Zb$ is not finite, but compact countable. In other words, each input cell is filled with a finite but unbounded amount of information.

This example naturally leads to a natural relaxation of~$\nar$, in which narrow functions are allowed to work on countable or compact countable alphabets on the input side. It would be interesting to investigate this class as well as other possible relaxations, and understand how much of the landscape is modified.

\section*{Acknowledgments}
\addcontentsline{toc}{section}{Acknowledgements}

We thank Guilhem Gamard, Emmanuel Jeandel, Julien Provillard and Alexis Terrassin for interesting discussions on the subject.

%\listoffigures

\newpage
\appendix

%SSSSS
\section{Background}\label{sec_back}
We state classical results that are used throughout the article.
%ssss
\subsection{Ramsey theory}
We will use two celebrated results from Ramsey theory. They can be found as Theorem 26.3 and 26.4 in \cite{Jukna11}.
\begin{theorem}[(Van der Waerden)]\label{thm_van_der_waerden}
Let~$C$ be a finite set and~$L\geq 1$. There exists~$N=N(C,k)$ such that every coloring~$c:[0,N]\to C$ admits a monochromatic arithmetic progression of length~$N$, i.e.~there exist~$a,\lambda\geq 1$ such that~$c$ is constant on~$\{a,a+\lambda,\ldots,a+(L-1)\lambda\}\subseteq [0,N]$.
\end{theorem}

The second result is a generalization in higher dimensions. A homothetic copy of a set~$F\subseteq \Z^m$ is~$u+\lambda F=\{u+\lambda f:v\in F\}$ for some~$u\in\Z^m,\lambda\in\Z$.
\begin{theorem}[(Gallai-Witt)]\label{thm_gallai_witt}
Let~$C$ be a finite set and~$c:\Z^m\to C$ a coloring. Every finite subset of~$\Z^m$ has a homothetic copy which is monochromatic.
\end{theorem}

%sss
\subsection{Group actions}
Let~$G$ be a group (we will only use~$G=(\Z^d,0,+)$ in this article). If~$H$ is a subgroup of~$G$, then~$G$ acts on the cosets~$G/H$ in a natural way (note that~$G/H$ is not necessarily a group because~$H$ may not be normal, but it is always a~$G$-set). The next result shows that every~$G$-action can be decomposed into a disjoint union of such actions.

\begin{theorem}[(Orbit-stabilizer theorem)]\label{thm_orbit_stab}
Let~$G$ be a group. Every~$G$-set~$X$ can be decomposed as~$X=\bigsqcup_{i\in I} G/H_i$ where~$I$ is a set,~$H_i$ is a subgroup of~$G$, and~$G$ acts on~$G/H_i$ in the natural way.
\end{theorem}
Indeed, the orbit of any~$x\in S$ is in bijection with~$G/G_x$ where~$G_x=\{g\in G:g\cdot x=x\}$ is the stabilizer of~$x$, and the orbits are pairwise disjoint.

If~$G,H$ are groups,~$X$ is a~$G$-set and~$Y$ is an~$H$-set, then a \textbf{homomorphism} from~$X$ to~$Y$ is a pair~$(f,\varphi)$ where~$f:X\to Y$ is a surjective function and~$\varphi:H\to G$ is a group homomorphism satisfying~$h\cdot f(x)=f(\varphi(h)\cdot x)$.

For instance, if~$X$ is a~$G$-set and~$\varphi:H\to G$ is a group homomorphism, then~$X$ is an~$H$-set with the action~$h\cdot x=\varphi(h)\cdot x$.

%sss
\subsection{Baire category}
\begin{theorem}[(Baire category theorem)]\label{thm_baire}
Let~$X$ be a completely metrizable topological space. If~$X=\bigcup_{n\in\N} X_n$ where each~$X_n$ is closed, then some~$X_n$ has non-empty interior.
\end{theorem}
We will use the following particular case. Let~$A$ be a finite alphabet,~$Z$ a countable set and~$X\subseteq A^Z$ be a closed set for the product topology. If~$X=\bigcup_{n\in\N}X_n$ where each~$X_n$ is a closed set, then there exists~$n\in\N$ and a finite pattern~$\pi$ such that~$\emptyset\neq [\pi]\cap X\subseteq X_n$.


\begin{thebibliography}{JdlRV06}

\bibitem[BDJ08]{BallierDJ08}
Alexis Ballier, Bruno Durand, and Emmanuel Jeandel.
\newblock Structural aspects of tilings.
\newblock In Susanne Albers and Pascal Weil, editors, {\em {STACS} 2008, 25th
  Annual Symposium on Theoretical Aspects of Computer Science, Bordeaux,
  France, February 21-23, 2008, Proceedings}, volume~1 of {\em LIPIcs}, pages
  61--72. Schloss Dagstuhl - Leibniz-Zentrum f{\"{u}}r Informatik, Germany,
  2008.

\bibitem[BFMS02]{PF02}
Valerie Berthe, S.~Ferenczi, C.~Mauduit, and A.~Siegel.
\newblock {\em {Substitutions in Dynamics, Arithmetics and Combinatorics}}.
\newblock Number 1794 in Lecture Note Mathematical Series. {Springer Verlag},
  2002.

\bibitem[BMP18]{BMP18}
Raimundo Brice{\~n}o, Kevin McGoff, and Ronnie Pavlov.
\newblock Factoring onto $\mathbb{Z}^d$ subshifts with the finite extension
  property.
\newblock {\em Proceedings of the American Mathematical Society},
  146(12):5129--5140, 2018.

\bibitem[Boy83]{Boyle83}
Mike Boyle.
\newblock Lower entropy factors of sofic systems.
\newblock {\em Ergodic Theory and Dynamical Systems}, 3(4):541–557, 1983.

\bibitem[BPS10]{BPS10}
Mike Boyle, Ronnie Pavlov, and Michael Schraudner.
\newblock Multidimensional sofic shifts without separation and their factors.
\newblock {\em Transactions of the American Mathematical Society},
  362(9):4617--4653, apr 2010.

\bibitem[DFP12]{DFP12}
Alberto Dennunzio, Enrico Formenti, and Julien Provillard.
\newblock Non-uniform cellular automata: Classes, dynamics, and decidability.
\newblock {\em Information and Computation}, 215:32--46, 2012.

\bibitem[Dol95]{Dolbilin95}
N.~Dolbilin.
\newblock The countability of a tiling family and the periodicity of a tiling.
\newblock {\em Discrete \& computational geometry}, 13(3-4):405--414, 1995.

\bibitem[dV13]{Vries13}
J.~de~Vries.
\newblock {\em Elements of Topological Dynamics}.
\newblock Mathematics and Its Applications. Springer Netherlands, 2013.

\bibitem[FH24]{FH24b}
Juliette Favereau and Mathieu Hoyrup.
\newblock Local generation of tilings: the even bicolor {W}ang tilesets.
\newblock Preprint, 2024.

\bibitem[Fur67]{Furstenberg67}
Harry Furstenberg.
\newblock Disjointness in ergodic theory, minimal sets, and a problem in
  diophantine approximation.
\newblock {\em Mathematical systems theory}, 1(1):1--49, Mar 1967.

\bibitem[Han74]{Hanf74}
William Hanf.
\newblock Nonrecursive tilings of the plane. {I}.
\newblock {\em The Journal of Symbolic Logic}, 39(2):283--285, 1974.

\bibitem[JdlRV06]{JRV06}
{\'E}lise Janvresse, Thierry de~la Rue, and Yvan Velenik.
\newblock A note on domino shuffling.
\newblock {\em The Electronic Journal of Combinatorics [electronic only]},
  13(1):R30, 2006.

\bibitem[JPS98]{JPS98}
William Jockusch, James Propp, and Peter Shor.
\newblock Random domino tilings and the arctic circle theorem, 1998.

\bibitem[Juk11]{Jukna11}
Stasys Jukna.
\newblock {\em Extremal Combinatorics - With Applications in Computer Science}.
\newblock Texts in Theoretical Computer Science. An EATCS Series. Springer,
  2011.

\bibitem[Lag09]{Lagae09}
Ares Lagae.
\newblock {\em Wang Tiles in Computer Graphics}.
\newblock Synthesis Lectures on Computer Graphics and Animation. Morgan {\&}
  Claypool Publishers, 2009.

\bibitem[LM95]{LM95}
Douglas Lind and Brian Marcus.
\newblock {\em An Introduction to Symbolic Dynamics and Coding}.
\newblock Cambridge University Press, 1995.

\bibitem[Mar79]{Marcus79}
Brian Marcus.
\newblock Factors and extensions of full shifts.
\newblock {\em Monatshefte für Mathematik}, 88:239--248, 1979.

\bibitem[Mye74]{Myers74}
Dale Myers.
\newblock Nonrecursive tilings of the plane. {II}.
\newblock {\em The Journal of Symbolic Logic}, 39(2):286--294, 1974.

\bibitem[Pro03]{Propp03}
James Propp.
\newblock Generalized domino-shuffling.
\newblock {\em Theoretical Computer Science}, 303(2):267--301, 2003.
\newblock Tilings of the Plane.

\bibitem[ST13]{ST13}
Ville Salo and Ilkka T\"{o}rm\"{a}.
\newblock Constructions with countable subshifts of finite type.
\newblock {\em Fundamenta Informaticae}, 126(2–3):263–300, April 2013.

\end{thebibliography}
\end{document}